\documentclass[12pt]{amsart}

\usepackage{amssymb,amsfonts, color,epsf}
\usepackage [cmtip,arrow]{xy}
\xyoption{all}
\usepackage {pb-diagram,pb-xy}
\usepackage{enumerate}
\usepackage{accents}
\usepackage[noautoscale]{youngtab}
\usepackage{subfigure}

\ifx\pdfpageheight\undefined
\usepackage[dvips,colorlinks=true,linkcolor=blue,citecolor=red,%
urlcolor=green]{hyperref}
\usepackage[dvips]{graphicx}
\makeatletter
\edef\Gin@extensions{\Gin@extensions,.mps}
\makeatother
\else
\usepackage[pdftex]{graphicx}
\usepackage[bookmarksopen=false,pdftex=true,breaklinks=true,%
backref=page,pagebackref=true,plainpages=false,%
hyperindex=true,pdfstartview=FitH,colorlinks=true,%
pdfpagelabels=true,colorlinks=true,linkcolor=blue,%
citecolor=red, urlcolor=darkgray, hypertexnames=false%
]%
{hyperref}
\fi
\usepackage{bm}

\usepackage{amssymb,color,epsf}
\usepackage{mathtools}
\usepackage{enumerate}
\usepackage{float}

\usepackage{amsmath}
\usepackage{tabularx}
\usepackage{color, colortbl}
\usepackage{url}

\DeclareMathAlphabet{\mathpzc}{OT1}{pzc}{m}{it}
\usepackage[margin=1.5in]{geometry}

\usepackage {pb-diagram,pb-xy}
\usepackage[mathscr]{eucal}
\usepackage[utf8]{inputenc}
\usepackage[T1]{fontenc}
\usepackage[english]{babel}
\usepackage{bm}
\usepackage{tabularx}
\usepackage{ltablex}

\usepackage{bm}

\usepackage{amsmath,amssymb,amsfonts}
\newtheorem{theorem}{Theorem}
\newtheorem{lemma}{Lemma}[section]

\newtheorem*{claim*}{Claim}

\newtheorem*{theorem*}{Theorem}
\newtheorem*{corollary*}{Corollary}
\theoremstyle{definition}
\newtheorem{definition}{Definition}[section]

\newtheorem{example}{Example}[section]

\newtheorem{notation}{Notation}[section]

\theoremstyle{remark}
\newtheorem{remark}[theorem]{Remark}

\definecolor{DarkBlue}{rgb}{0,0.1,0.55}

\numberwithin{equation}{section}

\newcommand {\hide}[1]{}

\newcommand {\junk}[1]{}

\newcommand {\R} {\mathrm{R}}

\newcommand {\Sphere}{\mbox{${\bf S}$}}     % Sphere
\newcommand {\Z}  {\mathbb{Z}}

\newcommand {\Q}         {\mathbb{Q}}

\newcommand {\eps} {{\varepsilon}}

\newcommand{\card}{\mathrm{card}}

\def\addots{\mathinner{\mkern1mu
		\raise1pt\vbox{\kern7pt\hbox{.}}
		\mkern2mu\raise4pt\hbox{.}\mkern2mu
		\raise7pt\hbox{.}\mkern1mu}}

\newcommand{\HH}  {\mbox{\rm H}}

\newcommand{\db}{\textsf{barcode-dist}}
\newcommand{\ddata}{\textsf{data-dist}}

\usepackage{algorithm}% http://ctan.org/pkg/algorithms
\usepackage{algpseudocode}
\algnewcommand\algorithmicinput{\textbf{Input:}}
\algnewcommand\INPUT{\item[\algorithmicinput]}
\algnewcommand\algorithmicoutput{\textbf{Output:}}
\algnewcommand\OUTPUT{\item[\algorithmicoutput]}

\algnewcommand\algorithmicproc{\textbf{Procedure:}}
\algnewcommand\PROCEDURE{\item[\algorithmicproc]}
\algnewcommand\algorithmiccomplexity{\textbf{Complexity:}}
\algnewcommand\COMPLEXITY{\item[\algorithmiccomplexity]}
\usepackage{multicol}
\usepackage{tikz-cd}
\usepackage{tabu}
\usepackage{caption}
\usepackage{booktabs}
\usepackage{makecell}
\newcolumntype{P}[1]{>{\centering\arraybackslash}p{#1}}

\newcommand{\Min}{\mathrm{Min}}
\newcommand{\barcode}{\mathrm{barcode}}

\makeatletter % for internal macros with @
\renewcommand\p@enumii{}
\makeatother
\title{Sequents, barcodes, and homology}
\author{Saugata Basu, Negin Karisani, Laxmi Parida}
\begin{document}
\begin{abstract}
    We consider the problem of generating hypothesis from data based on ideas from logic. We introduce a notion of barcodes, which we call sequent barcodes, that mirrors the barcodes in persistent homology theory in topological data analysis. We prove a theoretical result on the stability of these barcodes in analogy with similar results in persistent homology theory. Additionally we show that our new notion of barcodes can be interpreted in terms of a persistent homology of a particular filtration of topological spaces induced by the data. Finally, we discuss a concrete application of the sequent barcodes in a discovery problem arising from the area of cancer genomics.
\end{abstract}	
\maketitle
	
	\section{Introduction}
	We introduce  a new method for generating causal hypotheses from
	data. Our method draws inspiration from classical logic (sequent calculus) on one hand and the relatively new discipline of topological data analysis (or more precisely the theory of persistent homology) on the other. Our data consists of a finite number of labeled sets of two types: 
	\begin{enumerate}[(a)]
	    \item 
	    the causal ones  which we label by $A, B, C, \ldots$ using letters from the beginning of the alphabet (these could for example represent genotypes of cell lines);
	    \item
	    the  effect ones which we label by $X,Y,Z, \ldots$ using letters from the end of the alphabet (these could for example represent phenotypes of cell lines). 
	\end{enumerate} 
	The data sets (both causal and effect) are all assumed to be subsets of some universal set (for example the set of available cell lines).
	
	The goal is to identify certain set hypotheses of the form 
	\[
	A \wedge B \wedge C \cdots \Longrightarrow X \vee Y \vee Z \cdots
	\]
	which are strongly suggested by the data.
	In the paper we will refer to these hypotheses as ``sequents'' in keeping with the terminology of Gentzen's sequent calculus (see for example \cite{Takeuti}).
	
	One basic problem that needs to be addressed when using mathematical techniques on
	data coming from applications is the issue of noise. 
	There are many techniques that are being used to filter noise from signal in data. However, one technique that is closest in spirit (and indeed a motivates our method) is the technique of persistent homology which is the bedrock of the field of topological data analysis.

	Persistent homology 
	has emerged as an important tool in data analysis for separating noise from signal in topologically minded applications \cite{Dey-Wang2021}.
	We will define precisely persistent homology later in the paper (see Section~\ref{sec:persistent}). For the purposes of this 
	introduction -- persistent homology theory associates to point cloud data coming from some metric space a ``barcode'' (a finite set of intervals parameterized by a parameter often identified with time). The bars which are long (or ``persistent'' in time) correspond to signals while those which are short correspond to noise.
	
	In the theory we develop in this paper using sequents we associate a similar barcode to the given data (though defined in a somewhat different way from barcodes in persistent homology theory). Each bar in the barcode will represent a sequent 
	and we postulate that the long bars correspond to sequents
	which are more significant -- based on the philosophy that these sequents are the \emph{strongest conclusions lasting the longest} amongst all possible sequents according to the data.   
	We interpret these sequents (strongest last longer)  corresponding to long bars as the signal in the data (in analogy with persistent homology applications).

%%\subsection{Stability}
A key property satisfied by the barcodes coming from persistent homology is that of stability.
Stability theorems \cite{Cohen-Steiner-et-al-2007, Cohen-Steiner-et-al-2009, Cohen-Steiner-et-al-2010} state that the persistent homology or its associated barcode
is stable under perturbations of the input data. This makes persistent homology useful in applications, where the data often comes from physical measurements with their attendant sources of error.
 There are many variations of stability results in the literature. 
We refer the reader to \cite[Chapter VIII]{Edelsbrunner-Harer} and 
\cite[\S 5.6]{Chazal-et-al} for a survey of these results.
We prove an analogous stability result (Theorem~\ref{thm:stability}) 
for sequence barcodes (under certain restrictions on the data) that mirror the corresponding results 
for persistent homology barcodes.

The sequent barcodes that we define in this paper originate in a filtration of the poset of 
all possible sequents induced by the data. Posets and simplicial complexes are very closely related. Every simplicial complex has an associated poset (namely the face poset), and every poset has an
associated simplicial complex (whose simplices correspond to the chains of the poset). Filtrations of a poset induce in an obvious manner a filtration of the associated simplicial complex. Thus it is natural to speculate whether the sequent barcode can be interpreted as the persistent homology barcode of the associated simplicial
filtration. We show that this is not the case (at least in the straight-forward way as described above).
However, we are able to establish  a mathematical connection between the two notions of barcodes. 
We prove (see Theorem~\ref{thm:persistent} below) 
that the sequent barcodes that we introduce in this paper 
are \emph{embedded}  in the persistent barcodes of a certain filtration of a topological spaces
that is naturally associated to the data. In this way we reconnect our sequent barcodes to the persistent barcodes of topological data analysis completing the circle.
    
Finally, we apply the method of sequent barcodes to certain datasets coming
    from an application in cancer genomics and show that the 
    hypotheses predicted by the long bars in the sequent barcodes
    matches to a great extent those that are produced by more standard methods.
    
    The rest of the paper is organized as follows. In Section~\ref{sec:sequents} we define 
    sequent barcodes, discuss some basic examples and state and prove a stability result giving theoretical justification for their use in generating hypotheses. In Section~\ref{sec:topology}, we give the necessary background on persistent homology and establish the connection between our sequent barcodes and the barcodes of persistent homology theory. Finally in Section~\ref{sec:application} we discuss an application of the ideas introduced in this paper to a problem in cancer genomics. 
    In the Appendix we give precise definitions of persistent homology and barcodes in order to make the paper self-contained.
	
	\section{Sequents and partial order on them}
	\label{sec:sequents}  
	We assume two finite sets of propositional variables 
	\[
	V_A = \{a, b, c, \ldots \},
	\]
	and 
	\[
	V_S = \{x,y,z,\ldots\}.
	\]
	
	A \emph{sequent} $\Gamma \vdash \Delta$ is specified by a pair
	$(\Gamma,\Delta)$ where $\Gamma \subset V_A$, $\Delta \subset V_S$.
	We denote the set of all sequents by $\mathcal{S} = \mathcal{S}(V_A,V_S)$.
	
	\begin{definition}[Partial order on $\mathcal{S}$]
	\label{def:partial-order-on-sequents}
		If $s = (\Gamma \vdash \Delta), s' = (\Gamma' \vdash \Delta')$ are two sequents we will 
		denote $s \preceq s'$ if $\Gamma \subset \Gamma'$ and $\Delta \subset \Delta'$. 
	\end{definition}
	
	\begin{remark}
	\label{rem:partial-order-on-sequents:1}
	We identify a sequent $s = (\Gamma \vdash \Delta)$ with the element
	\[
	\bigvee_{a \in \Gamma} \neg a \vee \bigvee_{x\in \Delta} x
	\]
	in the free Boolean algebra $\mathcal{B}(V)$ generated by $V =V_A \cup V_S$.
	\end{remark}
	
	\begin{remark}
	\label{rem:partial-order-on-sequents:2}
		We should think of the partial order as saying that $s \preceq s'$, if
		from the sequent $s$ alone it is possible to deduce $s'$ (in sequent calculus terminology this means being able to deduce $s'$ from $s$ by successive applications of the (left and right) weakening and interchange rules). 
	\end{remark}
	
	We will define (depending on the data maps to be introduced below) two kinds of  
	\emph{filtrations}  of the sequent poset $\mathcal{S}$ introduced above. To each such filtration
	we will associate a barcode. In order to define these barcodes we need a few preliminary definitions.
	
	\begin{definition}[Minimal elements of a poset]
	\label{def:min-poset}
		Given any poset $(P,\preceq)$ we denote by $\Min(P)$ the set of minimal elements of $P$.
	\end{definition}
	
	\begin{definition}[Filtrations of posets]
	\label{def:filtration-of-poset}
		We call an indexed family 
		\[
		\mathcal{F} = (P_t)_{t \in [0,1]}
		\]
		of posets to be a filtration of posets if for each 
		$t,t', 0 \leq t \leq t' \leq 1$,
		$P_t \subset P_{t'}$.
	\end{definition}
	
	\begin{definition}[Barcodes of filtrations of posets]
	\label{def:barcode}
		Let $\mathcal{F} = (P_t)_{t \in [0,1]}$ be a filtration of posets.
		For $p \in P = P_1$, we denote by 
		\[
		I_p(\mathcal{F}) = \{t \in [0,1] \mid p \in \Min(P_t) \}.
		\]
		
		We denote 
		\[
		\barcode(\mathcal{F}) = \{(p, I_p(\mathcal{F})) \mid p \in P, I_p \neq \emptyset\},
		\]
		and call it the \emph{barcode of the filtration $\mathcal{F}$}.
	\end{definition}
	
	\begin{remark}
	\label{rem:barcode}
	    Note that in Definition~\ref{def:barcode}, 
	    for $p \in P$, if $I_p$ is non-empty then it is
	    equal to a half-closed interval $[s,t)$. In this case we will refer to
	    $s$ as the ``birth-time'' and $t$ as the ``death time'' of $p$ in the filtration.
	\end{remark}
	
	%%Given a subposet $\mathcal{S'} \subset \mathcal{S}$, we denote by 
	%%$\Min (\mathcal{S'})$ the set of minimal elements of $\mathcal{S}'$
	%%with respect to the partial order $\preceq$.
	
	\hide{
		We recall some facts about Boolean algebras.
		
		\begin{definition}[Boolean algebra, Free Boolean Algebra, Ideals, quotients]
		\end{definition}
		
		Let $\mathcal{B} = \mathcal{B}(V)$ be the free Boolean algebra on $V$. Let $\mathcal{S}'$
		be a set of sequents considered as a subset of $\mathcal{B}$, and 
		$I_{\mathcal{S}'}$ be the ideal generated by 
		$\neg \mathcal{S}' = \{\neg s \mid s \in \mathcal{S}'\}$.
	}

\subsection{Data maps}
Till now the discussion of the sequent poset above had nothing to do with 
data. We now introduce certain maps which we call 
data maps, which connects the sequent posets with the data.
This will also allow us to define certain filtrations on the sequent poset
(depending on the data maps).
For this purpose, we first fix a measure space $(X,m)$, with $m(X) = 1$, and let
	$\mathfrak{M}$ denote the set of measurable subsets of $X$. 
	For instance $X$ could be a finite set with the counting measure
	normalized so that $m(X) = 1$. In this case $\mathfrak{M} = 2^X$. 
	
\begin{definition}[Data maps]
	We will call a map 
	$$
	\displaylines{
		[[\cdot]]:V_A \cup V_S \rightarrow \mathfrak{M}, 
		%%[[\cdot]]:V_S \rightarrow \mathfrak{M}.
	}
	$$
	that maps each element of $V_A,V_S$ to a measurable subset of $\mathfrak{M}$ 
	a \emph{data map}.
	\end{definition}

A data map allows to define filtrations of the sequent poset and thus associate
barcodes.
		
	\subsection{Filtration of the poset of sequents and their barcodes}
	
	\begin{definition}[Filtration in sequent space]
		A filtration  
		of the poset of sequents is a filtration of posets
		$\mathcal{F} = (\mathcal{S}_t)_{t \in [0,1]}$ 
		where each $\mathcal{S}_t$ is a subposet of the poset $\mathcal{S}$.
	\end{definition}

	We introduce two different filtrations on the poset of sequents. 
	Depending on the application it might be preferable to choose one over the other -- 
	however, in this paper we do not discuss such choices.
	%%The following two filtrations  will play a key role in what follows.

	\begin{definition}[Fuzzy inclusions]
	\label{def:fuzzy}
		Let $S, T \in \mathfrak{M}$. 
		For $0 \leq t \leq 1$. We define the 
		predicate $S \subset_t T$ by 
		\[
		S \subset_t T \Longleftrightarrow \frac{m(S \cap T)}{m(S)} \geq t.
		\]
		Note that $\subset_t$ is not  necessarily transitive.
	\end{definition}
	
	\begin{remark}
	\label{rem:fuzzy}
	    Note that fuzzy set theory is a well-established field of research. In set theory there are 
	    two basic predicates -- membership and equality. In fuzzy set theory the membership relation is allowed to be fuzzy. In Definition~\ref{def:fuzzy} on the other hand we let the inclusion relationship (which generalizes equality) to be fuzzy. 
	    It is interesting to note that in 
	    \cite[Section 15.6]{Barr-Wells} the authors make the remark -- "In topos theory, both (membership and equality) may be fuzzy, but in fuzzy set theory, only membership is allowed to be". Our notion of fuzziness as defined above thus straddles the two worlds of fuzzy set theory and topos theory. 
	\end{remark}

	\begin{definition}[Filtration I]
		For $0 \leq t \leq 1$, and a data map 
		$[[\cdot]]:V_A \cup V_S \rightarrow \mathfrak{M}$,
		we define 
		\[
		\mathcal{S}'_t([[\cdot]]) = \left\{ (\Gamma \vdash \Delta) \in \mathcal{S} \mid
		\bigcap_{a \in \Gamma} [[a]] \subset_{1-t} \bigcup_{x \in \Delta} [[x]] \right\}.
		\]
		
		Note that 
		\[
		\mathcal{S}'_1([[\cdot]]) = \mathcal{S},
		\]
		and for $0 \leq t \leq t'\leq 1$
		\[
		\mathcal{S}'_{t}([[\cdot]]) \subset \mathcal{S}'_{t'}([[\cdot]]).
		\]
		
		We denote the filtration $(\mathcal{S}'_t([[\cdot]]))_{t \in [0,1]}$ by $\mathcal{F}'([[\cdot]])$.
	\end{definition}
	
	\begin{definition}[Filtration II] \label{def:filtrationII}
	For $0 \leq t \leq 1$, and a data map 
		$[[\cdot]]:V_A \cup V_S \rightarrow \mathfrak{M}$,
		we define 
		\[
		\mathcal{S}''_t([[\cdot]]) = \left\{ (\Gamma \vdash \Delta) \in \mathcal{S} \mid
		m\left(
		\bigcup_{a \in \Gamma} [[a]]^c \cup  \bigcup_{x \in \Delta} [[x]]
		\right) 
		\leq 1-t\right\},
		\]
		where for any $Y \in \mathfrak{M}$, $Y^c = X - Y \in \mathfrak{M}$.
		
		Note again that 
		\[
		\mathcal{S}''_1([[\cdot]]) = \mathcal{S},
		\]
		and for $0 \leq t \leq t'\leq 1$
		\[
		\mathcal{S}''_{t}([[\cdot]]) \subset \mathcal{S}''_{t'}([[\cdot]]).
		\]
		We denote the filtration $(\mathcal{S}''_t)_{t \in [0,1]}$ by $\mathcal{F}''([[\cdot]])$.
	\end{definition}
	
	The barcodes $\barcode(\mathcal{F}'([[\cdot]]))$, 
	$\barcode(\mathcal{F}''([[\cdot]]))$, 
	give important information
	about the implicational information contained  in the data 
	(i.e. in the data map $[[\cdot]]$). 
	
	\subsection{Example}
	We discuss an example with $V_A = \{a,b\}$, $V_S = \{x,y\}$,
	$X = \{1,2,3,4,5\}$ with the counting measure and $[[\cdot]]_A, 
	[[\cdot]]_S$ defined below.
	\[
	    [[a]]_A = \{1,2,3,4\}, \ \ \ \ \   [[b]]_A = \{2,3,4,5\}, \] \vspace{-2cm}
	 \[
	    [[x]]_S = \{1,2,3\}, \ \ \  \ \  \ \ \ \ \ \  [[y]]_S = \{3,4,5\}.
	\]
	The poset $\mathcal{S}(V)$ along with the ``birth times'' of each sequent
	in $\mathcal{S}(V)$ is shown in Figure~\ref{fig:poset-of-sequents}.
	\begin{figure}[h]
		\begin{center}
			\includegraphics[scale=0.85]{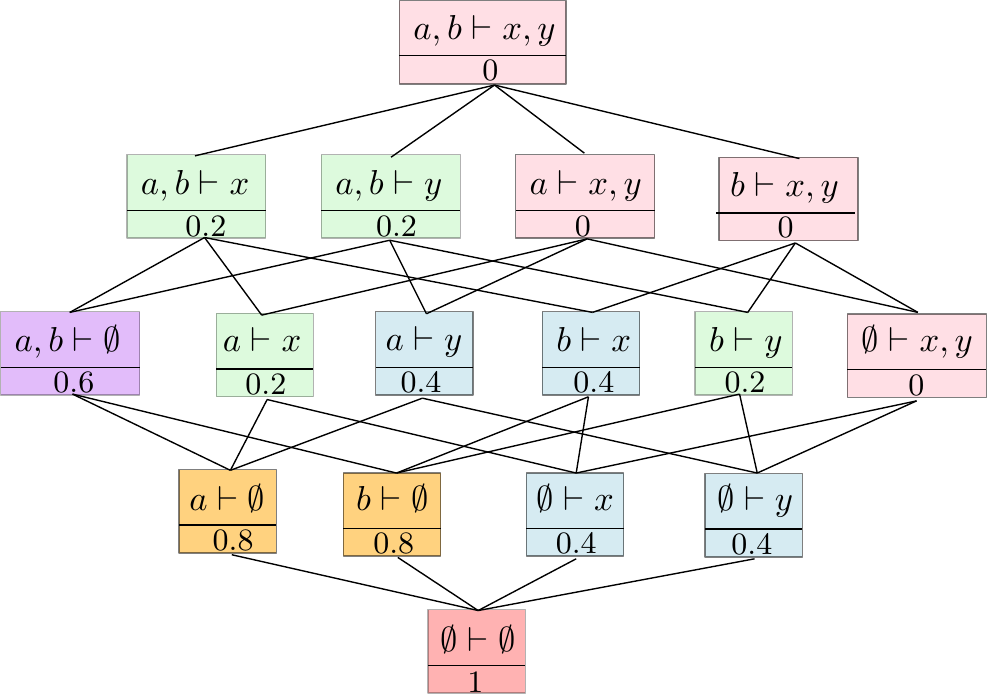}
			\caption{\small The poset $\mathcal{S}(V)$.}
			\label{fig:poset-of-sequents}
		\end{center}
	\end{figure}
	
	The sub-posets that appear in the filtration $\mathcal{F}''([[\cdot]])$ (Filtration II) induced by the above data with
	the minimal elements marked in black rectangles are shown in Figure~\ref{fig:filtration}.
	
	\begin{figure}[H]
		\centering
		\subfigure[ {\footnotesize $t=0$}]{
			\label{fig:poset:filtrationt0}
			\includegraphics[scale=0.7]{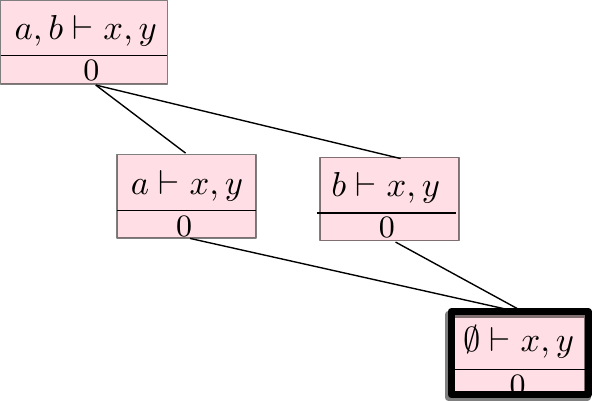}
		} 
		\hspace*{.2cm}
		\subfigure[{\footnotesize $t=0.2$}]{
			\label{fig:poset:filtrationt02}
			\includegraphics[scale=0.77]{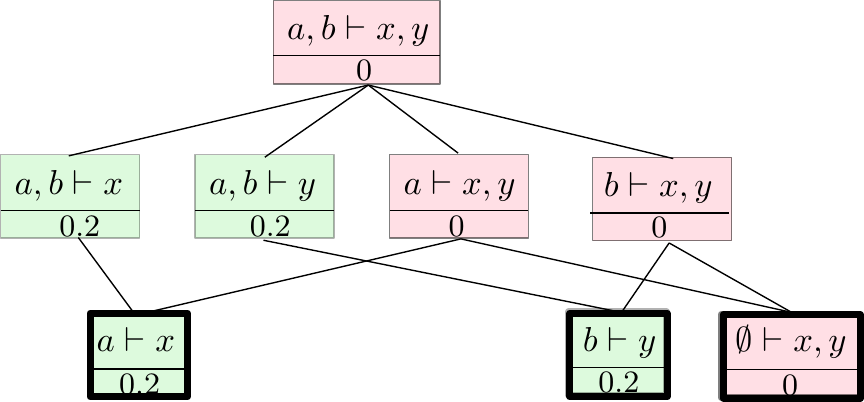} 
		}
		
		\subfigure[{\footnotesize $t=0.4$}]{
			\label{fig:filtrationt04}
			\includegraphics[scale=0.7]{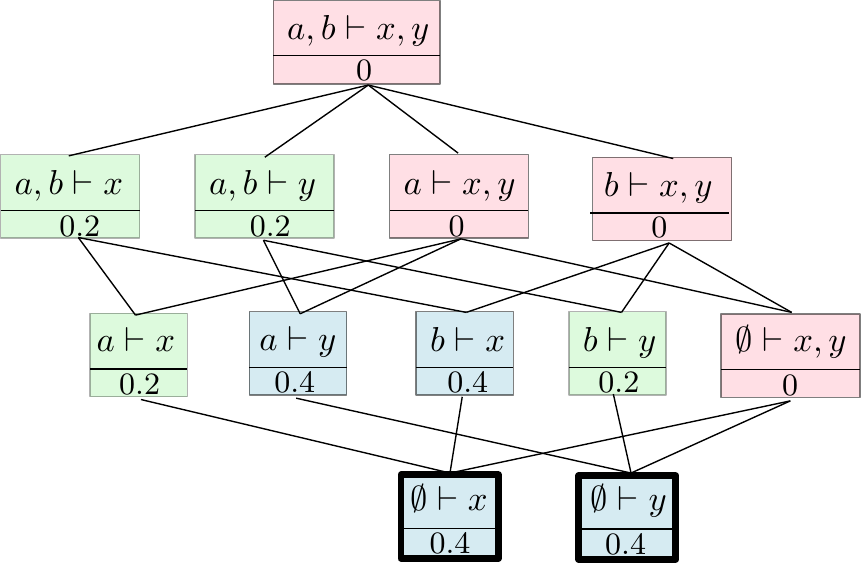} 
		}
		\subfigure[{\footnotesize $t=0.6$}]{
			\label{fig:filtrationt06}
			\includegraphics[scale=0.7]{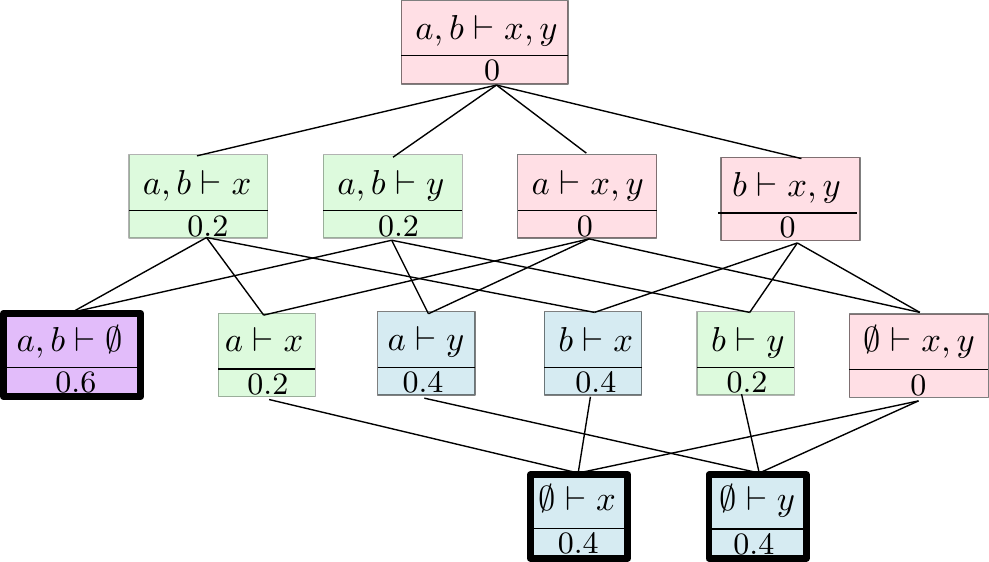} 
		}
		\subfigure[{\footnotesize $t=0.8$}]{
			\label{fig:filtrationt08}
			\includegraphics[scale=0.77]{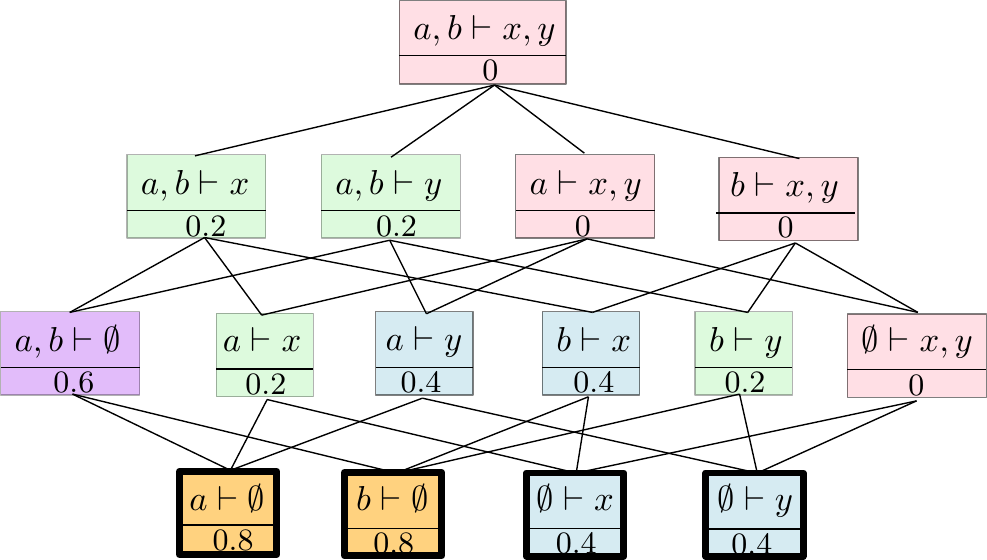} 
		}
		\subfigure[{\footnotesize $t=1$}]{
			\label{fig:filtrationt1}
			\includegraphics[scale=0.77]{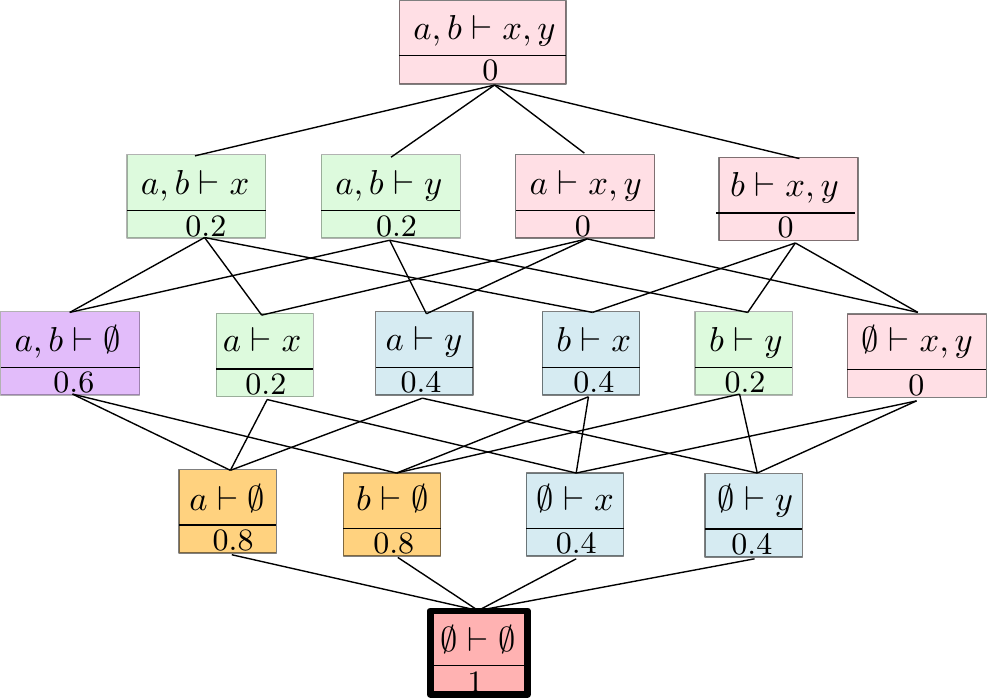} 
		}
		\vspace{-0.4cm}
		\caption{\small Sub-posets $\mathcal{S}''_t([[\cdot]])$ in the filtration when $t$ varies from 0 to 1.} \label{fig:filtration}
	\end{figure}

	%	\begin{figure}[H]
		%		\begin{center}
			%			\includegraphics[scale=0.75]{filtration1.pdf}%
			%			\caption{\small The $\mathcal{S}''_t$ for $t=0,1/5$.}
			%			\label{fig:filtration1}
			%		\end{center}
		%	\end{figure}
	%	
	%	\begin{figure}[H]
		%		\begin{center}
			%			\includegraphics[scale=0.75]{filtration2.pdf}%
			%			\caption{\small The The $\mathcal{S}''_t$ for $t=2/5,3/5$.}
			%			\label{fig:filtration2}
			%		\end{center}
		%	\end{figure}
	%	
	%	\begin{figure}[H]
		%		\begin{center}
			%			\includegraphics[scale=0.75]{filtration3.pdf}%
			%			\caption{\small The $\mathcal{S}''_t$ for $t=4/5,1$.}
			%			\label{fig:filtration3}
			%		\end{center}
		%	\end{figure}
	%	
	Finally, we show the $\barcode(\mathcal{F}''( [[\cdot]])$ in
	Figure~\ref{fig:barcode-example}. 
	\begin{figure}[H]
		\begin{center}
			\includegraphics[scale=0.6]{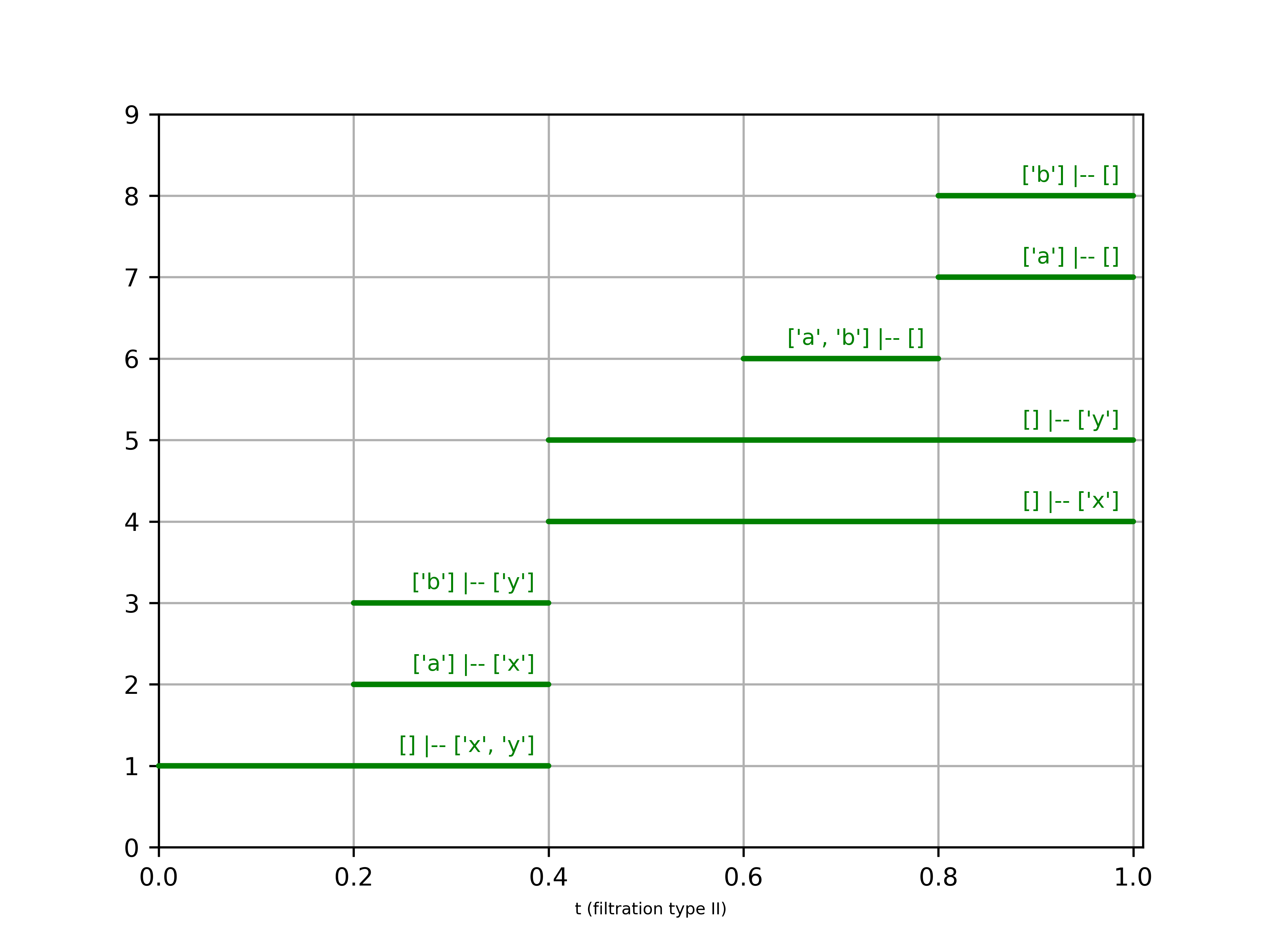}
			\caption{\small The $\barcode(\mathcal{F}''([[\cdot]]))$.}
			\label{fig:barcode-example}
		\end{center}
	\end{figure}

	\subsection{Stability of sequent barcodes}
	As mentioned in the Introduction,
	a very important property of barcodes arising in the theory of persistent homology is that of stability. The stability property implies that small changes in the data cases only small 
	changes in the barcode. This property is key to the utility of barcodes in practice. 
	In this section we discuss an analogous property for sequent barcodes.
	In the following discussion we only consider Filtration $I$ and omit $I$ from the subscript.
	
	The stability property for sequent barcodes should imply that 
	with a properly defined notion of distance between two sequent barcodes
	and that between two data maps, one can bound the former in terms of the latter. With this in mind, 
	we first define notion of distances between sequent two barcodes corresponding to two filtrations and also a distance between two data maps.
	
	\subsubsection{Distance between barcodes of two different filtrations of a fixed poset}
	\begin{definition}[Distance between sequence barcodes of two filtrations]
	\label{def:barcode-distance}
	 Suppose $\mathcal{F} = (P_t)_{t \in [0,1]},\mathcal{F}' = (P_t')_{t \in [0,1]}$ be two filtrations of a poset $P$. 
	 %%and let 
	%%$b = \barcode(\mathcal{F}), b' = \barcode(\mathcal{F}')$.
	We denote
	\[
	\db(\mathcal{F},\mathcal{F}') = \sum_{p \in P} |I_p(\mathcal{F}) \Delta I_p(\mathcal{F}')|,
	\]
	$\cdot \Delta \cdot$ denotes symmetric difference.   
	\end{definition}
	
	\begin{definition}[Distance between two data maps]
	\label{def:dist-data}
	 Given 
	$[[\cdot]]: V_A \cup V_S \rightarrow \mathfrak{M}, 
	[[\cdot]]': V_A \cup V_S \rightarrow \mathfrak{M},
	$
	we denote
	\[
	\ddata([[\cdot]],[[\cdot]]') = \sum_{v \in V_A\cup V_S} m([[v]] \Delta [[v]]')  
	.
	\]  
	\end{definition}
	
	\subsubsection{Assumptions}
	We will prove our stability result under certain assumptions on the 
	data maps which will be usually satisfied in applications. 
	
	The first assumption on the data map is the property of being $\eps$-granular.
	
	\begin{definition}[$\eps$-granular]
	 For $0 <\eps \leq 1$, we will say that a data map: $[[\cdot]]:V_a \cup V_S \rightarrow \mathfrak{M}$
	 is $\eps$-granular if
		for each $a \in V_A, x \in V_S$, $m([[a]]),m([[x]]) > \eps$. 
	\end{definition}

The second assumption that we need is to restrict our attention to those
sequents $\Gamma \vdash \Delta$, whose antecedent $\Gamma$ is sufficiently
significant given the data. This is quantified in the following definition.

\begin{definition}
 For $0 \leq \delta \leq 1$, 
	and a given map: $[[\cdot]]:V_A \cup V_S \rightarrow \mathfrak{M}$,
	we denote by
	$\mathcal{F}_{\delta}([[\cdot]])$ 
	the sub-filtration of $\mathcal{F}'([[\cdot]])$ in which only sequents
	$s = (\Gamma \vdash \Delta)$, with
	\[
	m(\bigcap_{b \in \Gamma}[[b]])/m([[a]]) > \delta
	\]
	appear for each $a \in \Gamma$.   
\end{definition}

We are now in a position to state and prove the following stability result.

	\begin{theorem}[Stability]
		\label{thm:stability}
		Let $0 < \eps,\delta \leq 1$.
		There exists a constant $C > 0$ (depending on $\card(V_A), \card(V_S), \eps, \delta$) such that for any two $\eps$-granular data maps,
		$[[\cdot]],[[\cdot]]': V_A \cup V_S \rightarrow \mathfrak{M},
		$
		such that the set of sequents appearing in $\mathcal{F}_\delta([[\cdot]])$ and 
		$\mathcal{F}_\delta([[\cdot]]')$ are the same,
		\[
		\db(\mathcal{F}_{\delta}([[\cdot]]),\mathcal{F}_{\delta} ([[\cdot]]')) \leq C \cdot  \ddata([[\cdot]],[[\cdot]]').
		\]
	\end{theorem}
	
	The proof of the theorem will follow from the following lemma.
	For a sequent $s = (\Gamma \vdash \Delta)$, 
	and 
	a map
	$[[\cdot]]: V_A \cup V_S \rightarrow \mathfrak{M}$, we denote
	\[
	t(s,[[\cdot]]) = \inf \{t \; \mid \;  s \in \mathcal{S}_{t}'([[\cdot]]\}.
	\]
	\begin{lemma}
		\label{lem:stability}
		With the same hypothesis as in Theorem~\ref{thm:stability},
		there exists a $C= C(\card(V_A \cup V_S, \eps,\delta)) > 0$ such that, 
		for any sequent $s = (\Gamma \vdash \Delta)$,
		\[
		|t(s,[[\cdot]]) - t(s,[[\cdot]]')| \leq C \cdot \ddata([[\cdot]],[[\cdot]]')
		\]
	\end{lemma}
	
	\begin{proof}
		Let 
		\[
		p = m(\bigcap_{a \in \Gamma}[[a]] \cap \bigcup_{x \in \Delta} [[x]]),
		\]
		\[
		q = 
		m(\bigcap_{a \in \Gamma}[[a]]),
	    \]
		so that 
		\[
		t(s,[[\cdot]]) = \frac{p}{q}.
		\]

		Similarly,
		let
		\[
		p' = m(\bigcap_{a \in \Gamma}[[a]]' \cap \bigcup_{x \in \Delta} [[x]]'),
		\]
		\[
		q' = 
		m(\bigcap_{a \in \Gamma}[[a]]'),
		\]
		so that 
		\[
		t(s,[[\cdot]]') = \frac{p'}{q'}.
		\]
		
		Let $\Delta p = p' -p, \Delta q = q'  -q$. Without loss of generality we can assume
		$\Delta q \geq 0$ (otherwise exchange roles of $s$ and $s'$).
		
		Observe that it follows from the hypothesis that 
		\begin{equation}
		\label{eqn:proof:lem:stability:1}
		|\Delta p|,|\Delta q | \leq \card(V_A \cup V_S) \cdot \ddata([[\cdot]],[[\cdot]]'),
		\end{equation}
		and also that
		\begin{equation}
		\label{eqn:proof:lem:stability:2}
		q \geq \delta \cdot \eps.
		\end{equation}
		
		Then,
		\begin{eqnarray*}
		\left| \frac{p}{q} - \frac{p'}{q'} \right| &=& \left| \frac{p}{q} - \frac{p+ \Delta p}{q + \Delta q} \right| \\
        &=&\left| \frac{p (\Delta q) - (\Delta p) q}{q (q + \Delta q)} \right| \\
        &=& \left| \frac{p}{q}\cdot \frac{\Delta q}{ q + \Delta q} - \frac{\Delta p}{q + \Delta q}\right| \\
        &\leq & \left|\frac{p}{q}\cdot \frac{\Delta q}{q + \Delta q }\right| + \left|  \frac{\Delta p}{q + \Delta q} \right| \\
        &\leq& \left|\frac{1}{q + \Delta q}\right| \left|(\Delta p + \Delta q) \right| \\
        &\leq & \frac{1}{q}(|\Delta p| + |\Delta q|).
		\end{eqnarray*}
		
		The lemma now follows from \eqref{eqn:proof:lem:stability:1} and
		\eqref{eqn:proof:lem:stability:2}.
		
		\hide{
		It follows from the hypothesis that $q,q' > \eps$. 
		
		Moreover, $|q-q'|,|p - p'| \leq \card(V_A \cup V_S) d([[\cdot]],[[\cdot]]')$.
		The lemma follows.
		}
		
	\end{proof}
	
	\begin{proof}[Proof of Theorem~\ref{thm:stability}]
	
	Let $\mathcal{S}_\delta$ denote the set of sequents appearing in 
	$\mathcal{F}_\delta([[\cdot]])$ (equivalently in $\mathcal{F}_\delta([[\cdot]])$).
	Recall from Definition~\ref{def:barcode-distance} that
	\[
	\db(\mathcal{F}_\delta([[\cdot]]),\mathcal{F}_\delta([[\cdot]')) = \sum_{s \in \mathcal{S}_\delta} |I_s(\mathcal{F}_\delta([[\cdot]])) \Delta I_s(\mathcal{F}_\delta([[\cdot]])')|.
	\]
	
	It follows from Lemma~\ref{lem:stability} that there exists a constant 
	$C= C(\card(V_A \cup V_S), \eps,\delta) > 0$ such that the left as well as the right endpoints of the intervals
	$I_s(\mathcal{F}_\delta([[\cdot]])), I_s(\mathcal{F}_\delta([[\cdot]]'))$ differ by at most $C$.
	The theorem follows.  
	\end{proof}

	\hide{
		\subsection{Another approach -- letting the binary relation $\preqeq$ on the sequents also depend on the data}
		For two sequents $s = (\Gamma \vdash \Delta), s' = (\Gamma' \vdash \Delta')$
		in $\mathcal{S}(\mathcal{D})$ 
		and $0 \leq t \leq 1$, we define the predicate
		$s \preceq_t s'$ by
		\[
		\bigcap_{a \in \Gamma} a \subset_t \bigcap_{a' \in \Gamma'} a' \wedge
		\bigcup_{b' \in \Delta'} b' \subset_t \bigcup_{b \in \Delta} b.
		\]
		
		Note that $s \preceq_1 s' \Leftrightarrow s \preceq s'$. For $t \neq 1$,
		$\preceq_t$ is not necessarily a partial order.
		
		We define a filtration $\mathcal{F}(\mathcal{D}) = (K_t(\mathcal{D}))_{t \in [0,1]}$ of simplicial complexes
		associated to $\mathcal{D}$ as follows.
		
		\begin{definition}
			\begin{eqnarray*}
				K_t(\mathcal{D})_0 &=&  \mathcal{S}_t(\mathcal{D}),
			\end{eqnarray*}
			and for $p > 0$, and $s_0,\ldots,s_p \in \mathcal{S}_t(\mathcal{D})$,
			\[
			[s_0,\ldots,s_p] \in K_t(\mathcal{D})_p
			\]
			if and only if 
			\[
			s_i \preceq_t s_j, \mbox{ for $0 \leq i,j \leq p$}.
			\]
		\end{definition}
	}
	
	\section{Filtrations posets and persistent homology}
	\label{sec:topology}
	In this section we establish a relationship between the barcode
	of a filtration of a poset with the persistent homology  barcodes of a filtration
	of a natural simplicial complex that we associate to the filtration
	of poset.
	
	We assume that the reader is familiar with the basic definition of 
	barcodes in the theory of persistent homology. All necessary background is included in the Appendix (Section~\ref{sec:persistent}).

	There is a very well studied and natural simplicial complex $\Delta(P)$
	(called the order complex of the poset) that is
	associated to a finite poset $P$. The simplices of $\Delta(P)$ corresponds to chains of the poset $P$. 
	
	\begin{example}
		Consider the poset whose Hasse diagram is shown below.
		\[
		\xymatrix{
			1 && 2 \\
			3 \ar[u]\ar[rru] &&  4 \ar[u]\ar[llu] \\
			5 \ar[u]\ar[rru] && 6 \ar[u]\ar[llu] 
		}
		\]
		The order complex $\Delta(P)$ is depicted in Figure~\ref{fig:sphere}.
		\begin{figure}[H]
			\begin{center}
				\includegraphics[scale=1]{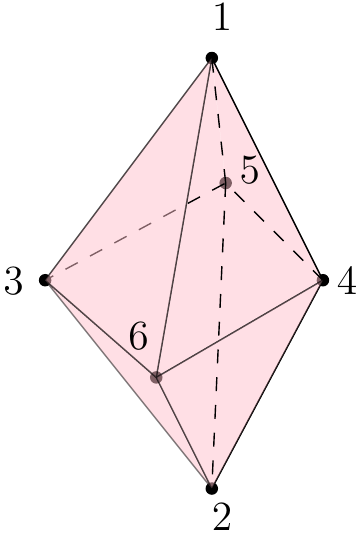}
				\caption{\small The order complex, $\Delta(P)$}
				\label{fig:sphere}
			\end{center}
		\end{figure}
	\end{example}
	
	A filtration $\mathcal{F} = (P_t)_{t \in [0,1]}$ gives rise to 
	a filtration of $\Delta(\mathcal{F}) = (\Delta(P_t))_{t \in [0,1]}$. 
	It is natural to try to establish a connection between 
	$\barcode(\mathcal{F})$ and $\barcode(\Delta(\mathcal{F}))$ where 
	for any filtration $(\Delta)_t)_{t \in [0,1]}$ we denote by
	$\barcode((\Delta_t)_{t \in [0,1]})$ the persistent homology barcode
	of the filtration $(\Delta_t)_{t \in [0,1]}$ (see Definition~\ref{def:barcode_multiplicity}).

	Unfortunately, the order complex $\Delta(P)$ of a poset $P$, while
	being a very useful topological object for studying properties of the poset $P$, the homology groups of $\Delta(P)$ does not retain enough information
	about the $\Min(P)$. Note that the sets $\Min(P_t), t \in [0,1]$ play a crucial role in the
	definition of $\barcode(\mathcal{F})$.
	
	For example consider the barcode of a filtration $\mathcal{F}$ of
	posets depicted in Figure~\ref{fig:barcode-poset}.
	
	\begin{figure}[h]
		\begin{center}
			\includegraphics[scale=.8]{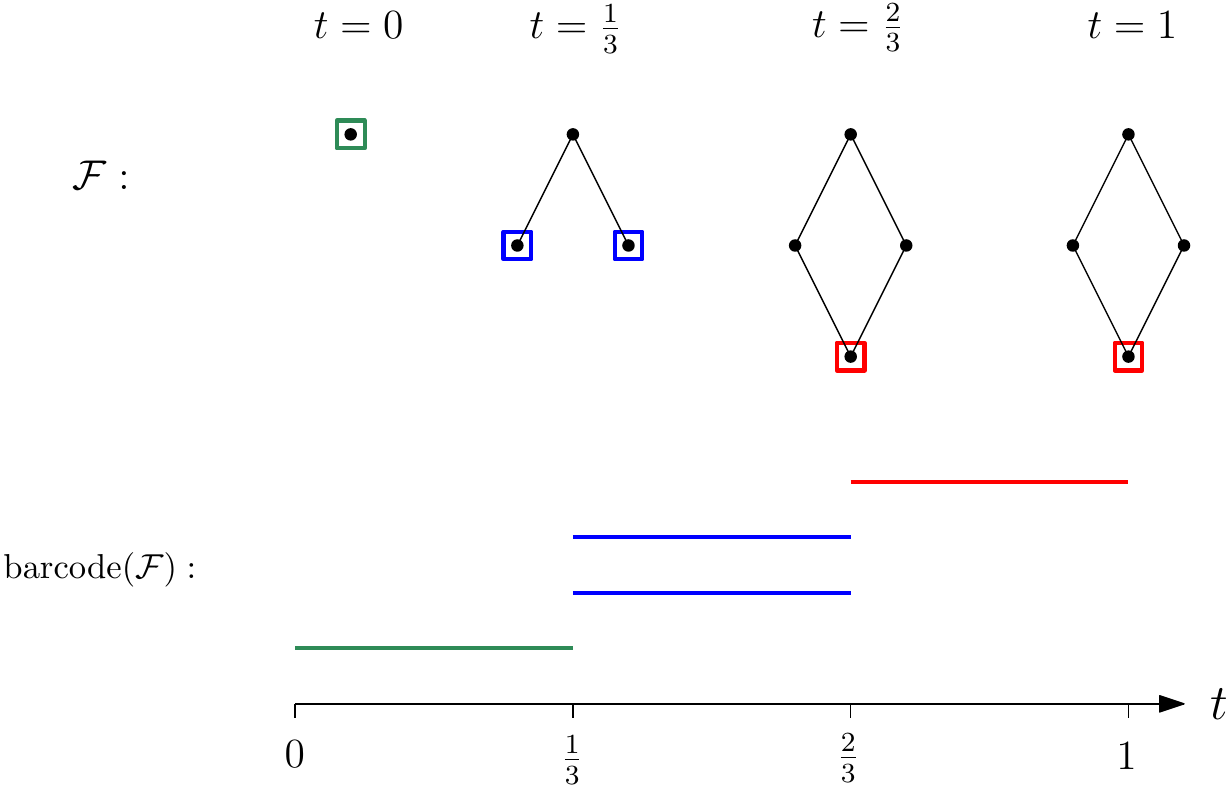}%
			\caption{\small Barcode of the poset filtration $\mathcal{F}$.}
			\label{fig:barcode-poset}
		\end{center}
	\end{figure}
	
	The barcode of the order complex of the poset filtration 
	$\mathcal{F}$ depicted above is 
	shown in Figure~\ref{fig:barcode-old}. Notice that this does not give any information about the
	barcode of the  poset filtration $\mathcal{F}$ (defined in
	Figure~\ref{fig:barcode-poset}).
	
	\begin{figure}[h]
		\begin{center}
			\includegraphics[scale=.8]{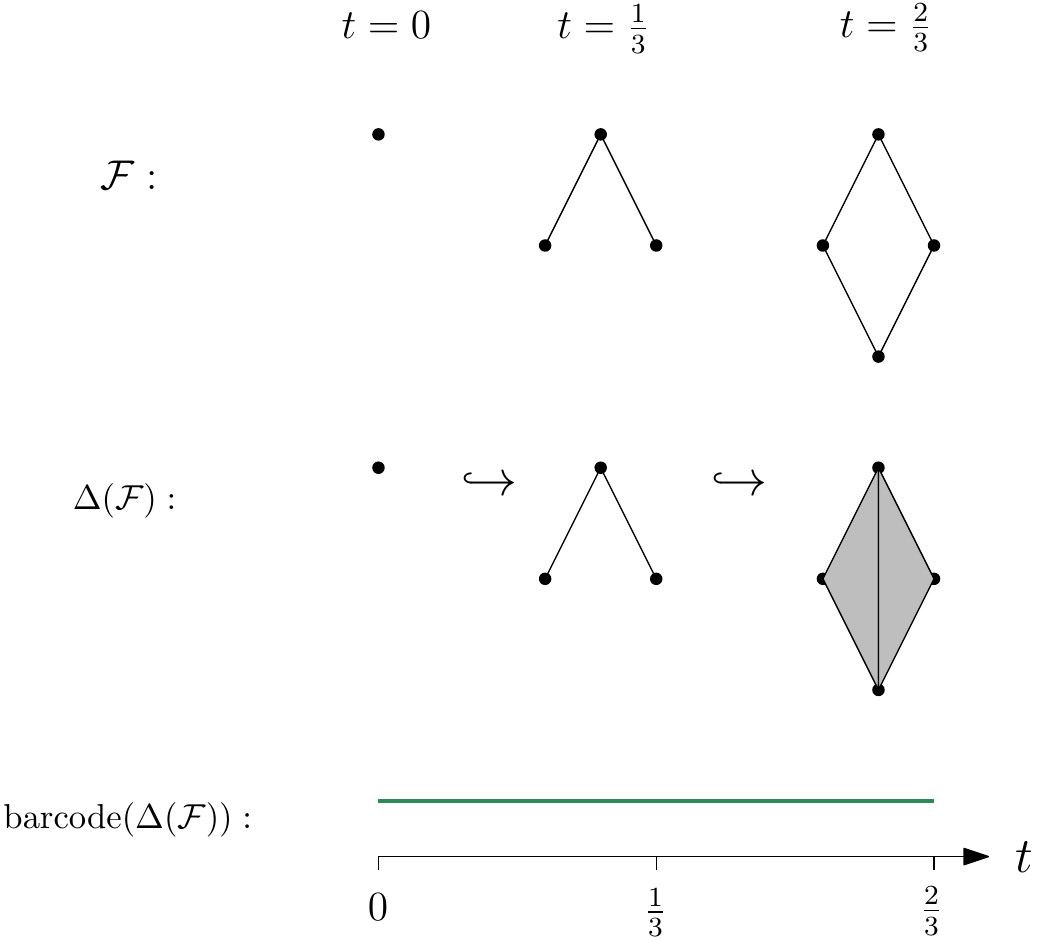}
			\caption{\small Barcode of the order complex of the poset filtration  $\Delta(\mathcal{F})$.}
			\label{fig:barcode-old} 
		\end{center}
	\end{figure}
	In order to remedy this deficiency of the order complex we introduce
	a variant of the order complex which we call denote by $\widetilde{\Delta}(P)$.
	
	The key difference between $\Delta(P)$ and $\widetilde{\Delta}(P)$ is that
	instead of including a simplex  of dimension $p$ for every chain $c_p$ in $P$
	of length $p+1$, we include a certain iterated based suspension,
	$\widetilde{\Delta}(c_p)$ that we define next.
	
	We define $\widetilde{\Delta}(c_p)$ for a chain 
	$c_p = (s_0 \prec s_1 \prec \cdots \prec s_p)$ of $P$ of length $p+1$
	by induction. Let for $0\leq i \leq p$, $c_i$ denote the chain
	$s_0 \prec \cdots \prec s_i$.
	
	For any based simplicial space $(X,x_0)$, 
	we denote by $\Sigma(X,x_0)$ the reduced suspension 
	\[
	X \times [0,1]/ X \times \{0\} \cup X \times \{1\} \cup \{x_0\} \times [0,1].
	\]
	The reduced suspension is a basic construction in algebraic topology 
	(see \cite{May}).
	
	\begin{remark}
	    We consider simplicial spaces since the reduced suspension as
	    defined above of a simplicial complex would in general be a simplicial space.
	\end{remark}
	
	\begin{definition}
		\begin{eqnarray*}
			\widetilde{\Delta}(c_0) &=& \Sigma(\emptyset \coprod \{s_0\}, \{s_0\}) \\
			&=& \{s_0\}; \\
			\widetilde{\Delta}(c_{i+1}) &=& \Sigma(\widetilde{\Delta}(c_i) \coprod \{s_{i+1}\}, \{s_{i+1}\}), 0 < i <p.
		\end{eqnarray*}
	\end{definition}

	Notice that for any sub-chain $c'_q = (s_{i_0} \prec \cdots \prec s_{i_q})$
	of the chain $c_p$
	there is a natural inclusion
	\begin{equation}
		\label{eqn:identifications}
		\widetilde{\Delta}(c'_q) \hookrightarrow \widetilde{\Delta}(c_p).    
	\end{equation}
	
	The simplicial sets $\widetilde{\Delta}(c_p)$ for $p=0,1,2$ is depicted below.
	
	\begin{figure}[H]
		\begin{center}
			\includegraphics[scale=0.75]{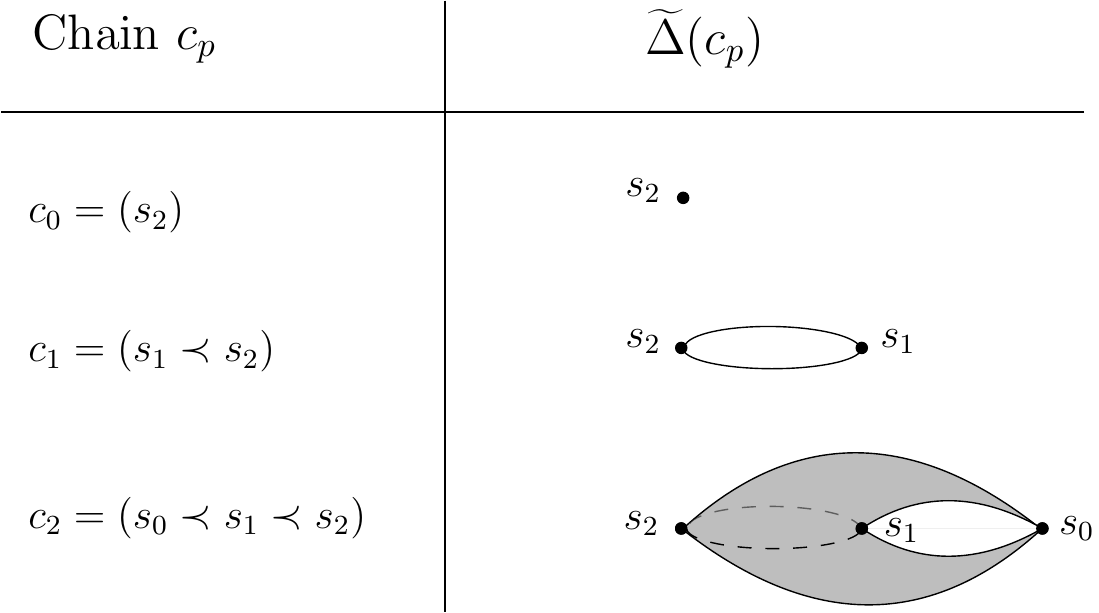}
			\caption{\small The simplicial sets $\widetilde{\Delta}(c_p)$ for $p=0,1,2$.}
			\label{fig:suspension-of-a-chain}
		\end{center}
	\end{figure}
	We now define
	\begin{definition}[$\widetilde{\Delta}(P)$]
		\label{def:suspension-of-a-poset}
		We define $\widetilde{\Delta}(P)$ as the union over all chains $c$ of 
		$P$ of the simplicial space $\widetilde{\Delta}(P)$ modulo the identifications implied by \eqref{eqn:identifications}
	\end{definition}
	
	Note that the space 
	$\widetilde{\Delta}(P)$ as defined above
	below will be a simplicial space (whose set of 
	$0$-simplices coincides with $P$) instead of
	a simplicial complex. We can obtain canonically a simplicial complex from $\widetilde{\Delta}(P)$  by taking a
	barycentric subdivision of each simplex in $\widetilde{\Delta}(P)$.

	Finally for
	a filtration $\mathcal{F} = (P_t)_{t \in [0,1]}$, we denote by $\widetilde{\Delta}(\mathcal{F})$
	the filtration $(\widetilde{\Delta}(P_t))_{t \in [0,1]}$ of the corresponding simplicial spaces defined in Definition~\ref{def:suspension-of-a-poset}.
	
	We now consider the persistent barcodes of the filtration 
	$\widetilde{\Delta}(\mathcal{F})$ (see Definition~\ref{def:barcode_multiplicity}).
	It is clear from the construction that if $p \in P$ is a minimal element in  poset $P$,
	then each chain $c$ of length $\ell$ with $p$ as its minimum element
	gives rise to a non-zero homology class in dimensions $0$ to $\ell-1$. Moreover, for
	an element $p \succ q$ in the poset $P$, with $p \neq q$, then each chain $c$ of $P$
	with $p$ as its minimum element is a subchain of some chain $d$ with $q$ as its minimum element. The linear maps 
	$\HH_i(\widetilde{\Delta}(c)) \rightarrow \HH_i(\widetilde{\Delta}(d))$ induced by
	the inclusion 
	$\widetilde{\Delta}(c) \hookrightarrow \widetilde{\Delta}(d)$ are $0$ 
	for $i > 0$. This observation leads to the following theorem.
	
	\begin{theorem}
	\label{thm:persistent}
		Let $(P,\preceq)$ be a finite poset and $\mathcal{F} = (P_t)_{t \in [0,1]}$ a filtration of $P$ by subposets $P_t$, such that $P_0 = \{p_0\}$, where $p_0$ is the unique maximal element of $P$. 
		%%For each $i=1,2$,  $[[\cdot]]: V_A \cup V_S \rightarrow \mathfrak{M}$,
		%%$s = (\Delta \vdash \Lambda)  \in \mathcal{S}$ such that
		%%$(\Delta,\Lambda) \neq (A,S)$, and $(s,I_s) \in \barcode(\mathcal{F}_i([[\cdot]])$ with $I_s = [t,t']$, 
		Then, for each
		$p  \in P - \{p_0\}$ such that
		$(p,I_p) \in \barcode(\mathcal{F})$ with $I_p = [t,t']$, 
		there exists $i \geq 0$ such that 
		$\mu^{t,t'}_i(\widetilde{\Delta}(\mathcal{F})) > 0$.
	\end{theorem}
	
	\begin{proof}
		Follows easily from the prior discussion.
	\end{proof}
	
	As illustration of Theorem~\ref{thm:persistent} we
	show the barcode $\barcode(\widetilde{\Delta}(\mathcal{F}))$ 
	for the same poset filtration $\mathcal{F}$ whose barcode is depicted in Figure~\ref{fig:barcode-poset}.
	Observe that 
	the barcode of the
	poset filtration $\mathcal{F}$ depicted in Figure~\ref{fig:barcode-poset}
 is \emph{embedded} in that of the persistent homology filtration $\widetilde{\Delta}(\mathcal{F})$.
 Figure~\ref{fig:barcode-old} shows that this is not true for the 
	filtration $\Delta(\mathcal{F})$ of the order complex of $P$.

	\begin{figure}[H]
		\begin{center}
			\includegraphics[scale=0.75]{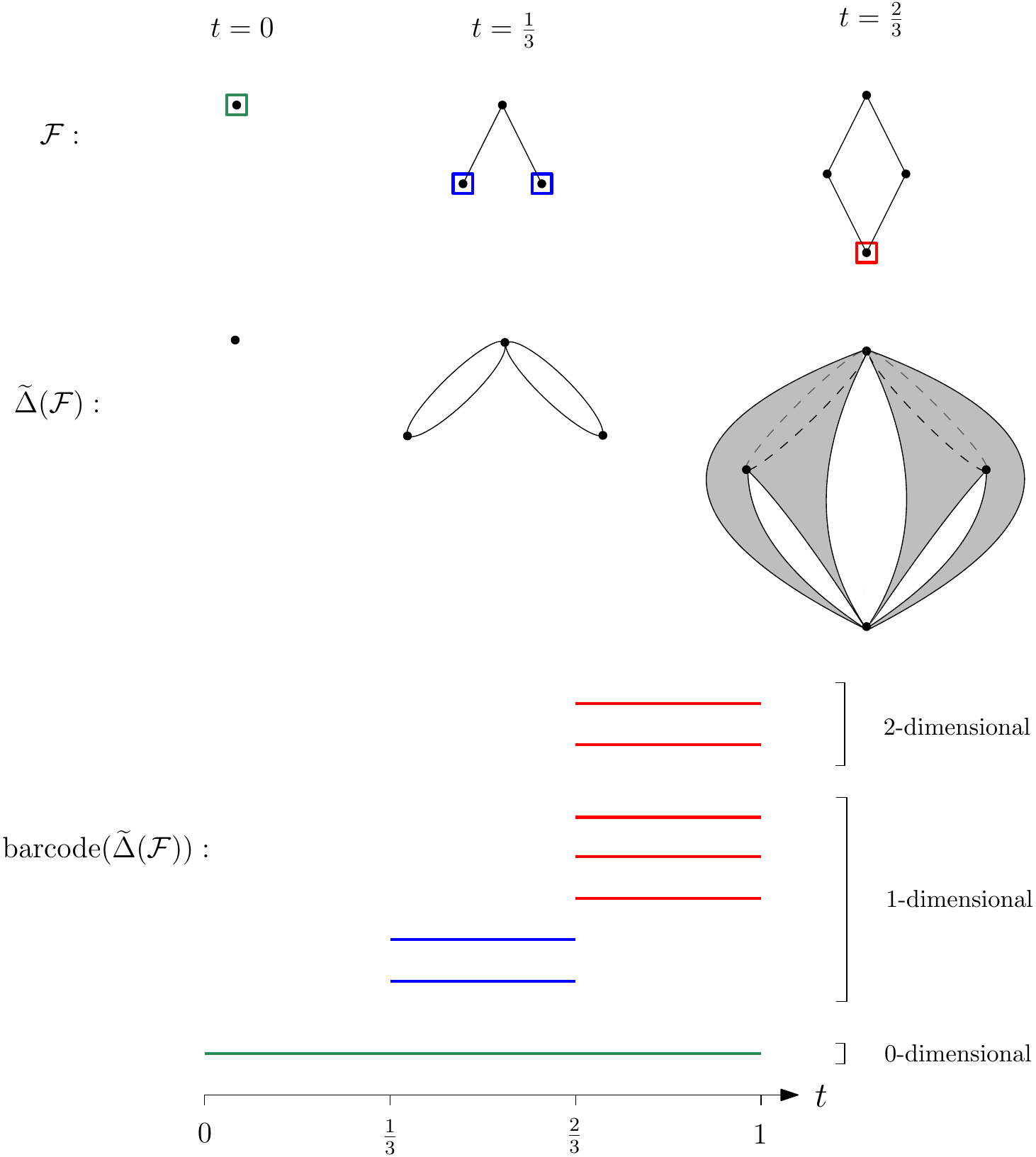}
			\caption{\small The barcode $\barcode(\widetilde{\Delta}(\mathcal{F}))$.}
			\label{fig:barcode-new}
		\end{center}
	\end{figure}

	\section{Application}
	\label{sec:application}
	
	In this section, we present an example using real-world 
	%%and synthetic 
	data to demonstrate the applications of our proposed method.
	%%%%%%%%%%%%%%%%%%%%%%%%%%%%%%%%%%%%%%%%%%%%%%%%%%%%%%%%%%%%%%%%%
	\subsection{Genetic dependencies in cancer cell lines}
	
	Identifying genes that are essential for the survival and proliferation of cancer cells is a key step in the development of target-based drugs. This process can be carried out by analyzing the genetic dependencies of in vitro cancer cell lines \cite{Lin2020}. However, models that are able to automatically predict gene dependencies are of a great importance to help scientists with this aim \cite{chuai_deepcrispr_2018, chiu_predicting_2021}. In a recent study~\cite{Dempster2020}, the authors systematically analyze the predictive power of a comprehensive set of multi-omics features across hundreds of CRISPR screens in cancer cell lines. They use random forest regressors to predict gene dependencies and to rank important features. Their results demonstrate that gene expression values are more predictive for cancer cell lines vulnerabilities than genomics features. However, they show that there are a few exceptions such as BRAF, KRAS, and NRAS whose gene dependencies are best predicted with their own mutation values. In this section, we examine the latter result using our proposed approach.
	
	\subsection{Dataset}
	We downloaded the data from DepMap portal\footnote{\url{https://depmap.org/}}. Gene dependency values and their probabilities are taken from the CRISPR\_gene\_effect and Achilles\_gene\_dependency files respectively in DepMap Public 21Q4. The corresponding mutation data are taken unaltered from the files listed below: \newline
	CCLE\_mutations\_bool\_hotspot, \newline CCLE\_mutations\_bool\_damaging, \newline CCLE\_mutations\_bool\_nonconserving. \newline
	Gene expressions and copy numbers also are taken from the CCLE\_expression and CCLE\_gene\_cn files respectively.
	This all amounted to 894 cell lines.
	
	\subsection{Method}
	
	We selected gene dependency probabilities of `HRAS (3265)', `KRAS (3845)', `BRAF (673)', and `NRAS (4893)' and  used a threshold of $0.5$ to create the binary labels. Therefore, we have the propositional variables $V_S$ and $V_A$ as follows.
	\\ 
	\\
	$V_S =$ \{\text{`HRAS (3265)', `KRAS (3845)', `BRAF (673)', `NRAS (4893)'}\}, 
	\begin{equation*}
		\begin{aligned}
			V_A =  & \ \{ \text{`HRAS (3265)\_Hot', `KRAS (3845)\_Hot', `BRAF (673)\_Hot', }\\
			& \text{ `NRAS (4893)\_Hot', `KRAS (3845)\_Dam', `BRAF (673)\_Dam', }\\
			& \text{`NRAS (4893)\_Dam', `HRAS (3265)\_NonCon', `KRAS (3845)\_NonCon',}\\
			& \text{`BRAF (673)\_NonCon', `NRAS (4893)\_NonCon'\}},
		\end{aligned}
	\end{equation*}
	\\
	\noindent where the suffixes \_Hot, \_Dam, and \_NonCon correspond to hotspot, damaging, and nonconserved mutations respectively. We used definition~\ref{def:filtrationII} to compute the filtration of the posets and obtained the associated barcode, shown in Figure~\ref{fig:barcode-gdp}.
	
	\begin{figure}[h]
		\centering
		\includegraphics[scale=.8]{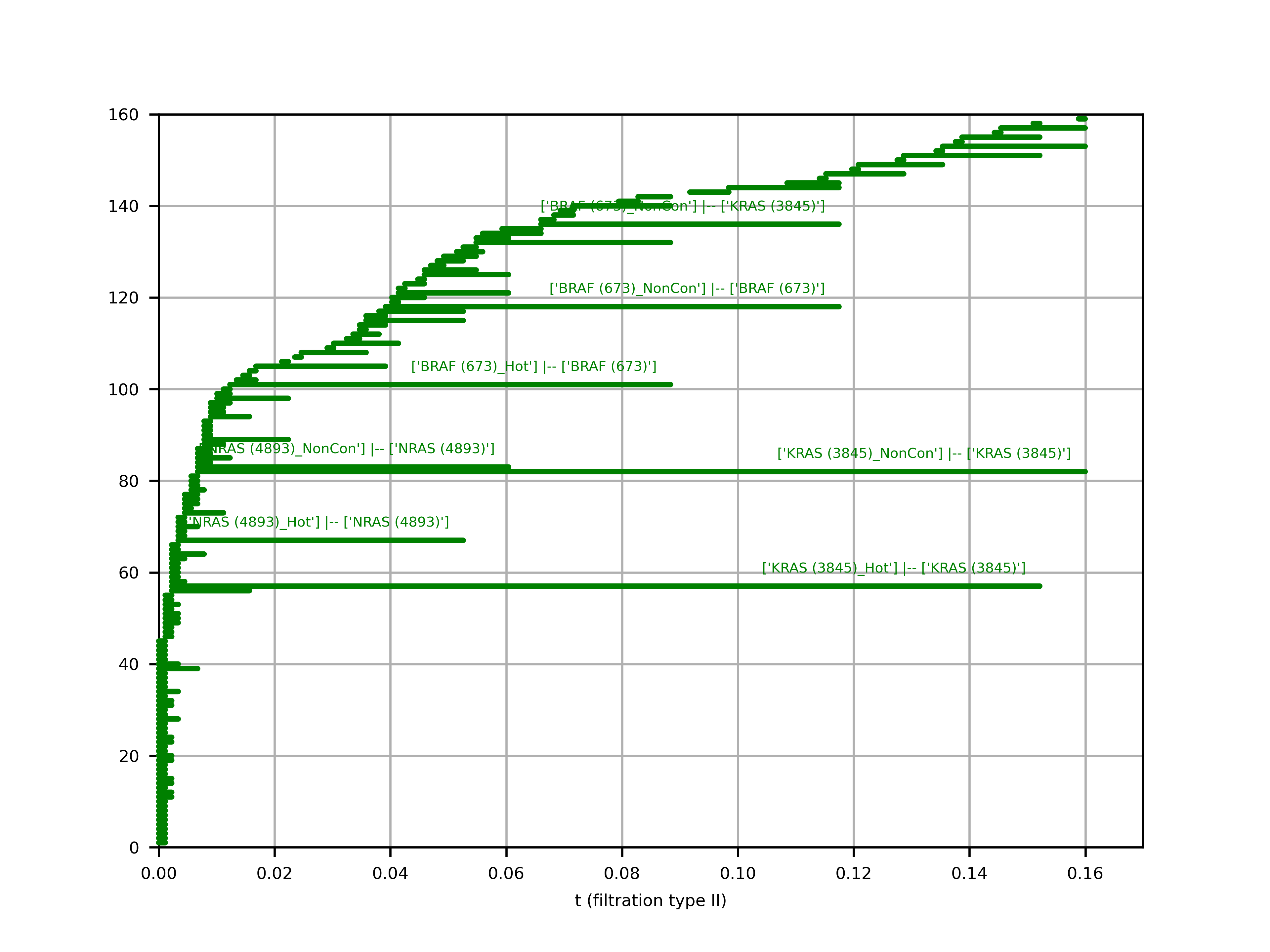}
		\caption{Barcode of the gene dependencies and their corresponding mutations}	
		\label{fig:barcode-gdp}
	\end{figure}
	
	To compare the results with those reported in \cite{Dempster2020}, we followed their pipeline to filter and train the regression models using expression values, copy numbers, and mutation data to predict the gene dependency values. We computed the Pearson correlation between the predicted values and the observed gene dependencies and obtained the top predictive features. Table~\ref{tab:model} reports the models' performance for the selected target genes\footnote{The reported results are similar to those published in \cite{Dempster2020}.}. 
	
	\begin{table}
		\captionsetup{justification=centering}
		\caption{\small Performance of the trained random forest regressors.}
		\label{tab:model} 
		\centering
		\begin{tabular}{p{2.7cm}P{2.5cm}}%p{0.5cm}
			\hline\noalign{\smallskip}
			\textbf{Gene} & \textbf{Pearson correlation}\\
			\noalign{\smallskip}
			\Xhline{3\arrayrulewidth}
			\noalign{\smallskip}
			`HRAS (3265)'  & 0.48\\
			`KRAS (3845)'  & 0.81\\
			`BRAF (673)'  & 0.79\\
			`NRAS (4893)' & 0.78\\
			\noalign{\smallskip}\hline\noalign{\smallskip}
		\end{tabular}
	\end{table}

	\subsection{Results and Discussion} Given a set $\Delta$, the longest bar with the sequent $\Gamma \vdash \Delta$, compared to the shorter bars with sequents $\Gamma' \vdash \Delta$, implies that the features in $\Gamma$  have the most decisive role in estimating the gene dependencies of the genes in $\Delta$. In Figure~\ref{fig:barcode-gdp}, for $\Delta = \{\text{`KRAS (3845)'}\}$ we have two long bars corresponded to the sequents
	
	\begin{equation*}
		\begin{aligned}
			[\text{`KRAS (3845)\_NonCon'}]  & \ \vdash & [\text{`KRAS (3845)'}], \\
			[\text{`KRAS (3845)\_Hot'}]  & \ \vdash & [\text{`KRAS (3845)'}].
		\end{aligned}
	\end{equation*}
	
	This implies that among all the elements of $V_A$, the most important features for `KRAS (3845)' are `KRAS (3845)\_Hot' and `KRAS (3845)\_NonCon'. Similarly, sequents corresponded to the top two longest bars for $\Delta = \{\text{`NRAS (4893)'}\}$ and those for $\Delta = \{\text{`BRAF (673)'}\}$ are listed below.
	
	\begin{equation*}
		\begin{aligned}
			[\text{`NRAS (4893)\_NonCon'}]  & \ \vdash & [\text{`NRAS (4893)'}], \\
			[\text{`NRAS (4893)\_Hot'}]  & \ \vdash & [\text{`NRAS (4893)'}].
		\end{aligned}
	\end{equation*}
	
	\begin{equation*}
		\begin{aligned}
			[\text{`BRAF (673)\_NonCon'}]  & \ \vdash & [\text{`BRAF (673)'}], \\
			[\text{`BRAF (673)\_Hot'}]  & \ \vdash & [\text{`BRAF (673)'}].
		\end{aligned}
	\end{equation*}
	
	If $\Delta = \{\text{`HRAS (3265)'}\}$, there is not a significant long bar in Figure~\ref{fig:barcode-gdp}. However, if we restrict the bars to those with $\Delta = \{\text{`HRAS (3265)'}\}$, among the longest bars, we have the sequents listed below. 
	
	\begin{equation*}
		\begin{aligned}
			[\text{`HRAS (3265)\_NonCon'}]  & \ \vdash & [\text{`HRAS (3265)'}], \\
			[\text{`HRAS (3265)\_Hot'}]  & \ \vdash & [\text{`HRAS (3265)'}],
		\end{aligned}
	\end{equation*}
	
	Comparing Figure~\ref{fig:barcode-gdp} and Table~\ref{tab:model} suggests that the length of a bar corresponded to the sequent $\Gamma \vdash \Delta$ could be a major factor, if not the only one, in estimating the predictive power of the features in $\Gamma$. In this example, `HRAS (3265)\_NonCon' for `HRAS (3265)' is not as predictive as `KRAS (3845)\_NonCon' for `KRAS (3845)'.

	Note that the important features mentioned above are consistent with the top features obtained from the trained random forest regressors \cite{Dempster2020}.

	%===================================================
	\bibliographystyle{amsplain}
	\bibliography{reference}

\appendix
\section{Simplicial complexes, filtrations, persistent homology and persistent barcodes}
\label{sec:persistent}

In this section we give some background on persistent homology for the reader's benefit.

\subsection{Background on homology theory}
Homology theory is a classical tool in mathematics for studying topological spaces. A particular homology theory (such as simplicial, singular, Čech, sheaf etc. \cite{Spanier}) 
associates to 
any topological space $X$, a certain vector space (assuming we take
coefficients in a field) graded by dimension, denoted 
\[
\HH(X) = \bigoplus_{p \geq 0}  \HH_p(X)
\]
in a functorial way
(maps of spaces give rise to linear maps of the corresponding vector spaces, and compositions are preserved). 
%%If the space $X$ is sufficiently nice -- for example, if $X$ is a compact subset of $\R^d$ defined in terms of a finite number of polynomial inequalities (i.e. $X$ is a compact \emph{semi-algebraic} subset of $\R^d$) --  then the various homology theories mentioned above produce isomorphic homology groups.
%%For such spaces 
The dimension of $\HH_p(X)$ (called the $p$-th Betti number of $X$) has a geometric meaning. It is equal to the number of independent $p$-dimensional  ``holes'' of $X$. In particular, the dimension of $\HH_0(X)$ is the number of connected components of $X$, and if $X$ is a graph embedded in $\R^d$, then
the dimension of $\HH_1(X)$ is the number of ``independent cycles'' of $X$.
The dimension of $\HH_1(X)$ is equal to $0$ if $X$ is the $2$-dimensional sphere, while it is equal to $2$ is $X$ is the torus, reflecting the fact that
the sphere does not have $1$-dimensional holes while the torus has two 
independent ones. Finally, the dimension of $\HH_2(X)$ equals $1$ for either the sphere or the torus indicating that they both have a $2$-dimensional hole. 

%%The passage from the topological space $X$, to the vector space $\HH(X)$, inevitably leads to a loss of 
%%topological information (non-homeomorphic spaces can have isomorphic homology groups). On the positive side,  the homology groups are usually computable (even efficiently computable) unlike other topological invariants like homotopy groups -- and they retain information about many features of interest of the original space. 

\subsection{Simplicial homology}
The homology theory that is relevant for this paper is \emph{simplicial homology theory}  which associates to any simplicial complex $K$
(defined below), its \emph{simplicial homology groups}, $\HH_*(K)$. 
%%We define simplicial complexes and their homology groups precisely later in the paper (see Definition~\ref{def:simplicial-complex} and \ref{def:simplicial-homology}). For the moment 
It is useful to
think of simplicial complexes as higher dimensional generalizations of graphs (with no multiple edges or self loops). A $p$-dimensional simplex is spanned
by $(p+1)$-vertices of the complex (just as an edge in a graph is spanned by two vertices). An edge in a graph is just a $1$-simplex.

\subsection{Simplicial complexes}
\begin{definition}
\label{def:simplicial-complex}
A \emph{finite simplicial complex} $K$ is a set of ordered subsets of $[N] = \{0,\ldots N\}$ for some $N \geq 0$, such that if $\sigma \in K$ and $\tau$
is a subset of $\sigma$, then $\tau \in K$. 
\end{definition}

\begin{notation}
\label{def:p-simplices}
If $\sigma = \{i_0,\ldots,i_p\} \in K$, 
with $K$ a finite simplicial complex,
and $i_0 < \cdots < i_p $, we will denote $\sigma = [i_0,\ldots,i_p]$ and call $\sigma$ a \emph{$p$-dimensional simplex of $K$}. We will denote by $K^{(p)}$ the set of $p$-dimensional simplices of $K$.
\end{notation}

A simplicial complex $K$ is a combinatorial object 
%%(cf. Definition~\ref{def:simplicial-complex}), 
but it comes with 
an associated topological space $|K|$ (the so called geometric realization of $K$), and the simplicial homology groups, $\HH_*(K)$, are actually  topological invariants of $|K|$ (two simplicial complexes having homeomorphic geometric realizations have isomorphic simplicial homology groups). The geometric realization functor $|\cdot|$ will play no role in the paper but is useful to keep in mind for visualization purposes (and also to connect simplicial homology of a simplicial complex to the notion of homology being a functor from the category of topological spaces to vector spaces as stated in the previous paragraphs).   

What is simplicial homology ?

Let $K$ denote a finite simplicial complex and for $p \geq 0$,   let $K^{(p)}$ denote the set of $p$-dimensional simplices of $K$.

\begin{definition}[Chain groups and their standard bases]
\label{def:chain-groups}
Suppose $K$ is a finite simplicial complex.
For $p \geq 0$,  we will denote by 
$C_p(K) = C_p(K,\R)$ (the $p$-th chain group),   
the $\R$-vector space generated by the elements of $K^{(p)}$,
i.e.
\[
C_p(K) = \bigoplus_{\sigma \in K^{(p)}}\Q \cdot \sigma.
\]
\end{definition}

\begin{definition}[The boundary map]
\label{def:boundary-map}
We denote by $\partial_p(K): C_p(K) \rightarrow C_{p-1}(K)$ the 
linear map (called the $p$-th \emph{boundary map})  defined as follows. Since 
$\left(\sigma \right)_{\sigma \in K^{(p)}}$ is a basis of $C_p(K)$ it is enough to define the image of each $\sigma \in C_p(K)$. 
We define for $\sigma = [i_0,\ldots,i_p] \in K^{(p)}$,
\[
\partial_p(K)(\sigma) = \sum_{0\leq j \leq p} (-)^j [i_0,\ldots, \widehat{i_j}, \ldots,i_p] \in C_{p-1}(K),
\]
where $\widehat{\cdot}$ denotes omission.
\end{definition}

One can easily check that the boundary maps $\partial_p$ satisfy the key property that 
\[
\partial_{p+1}(K) \;\circ\; \partial_{p}(K) = 0,
\]
or equivalently
that 
\[
\mathrm{Im}(\partial_{p+1}(K)) \subset \ker(\partial_p(K)).
\]

\begin{notation}[Cycles and boundaries]
We denote 
\[
Z_p(K) = \ker(\partial_p(K)),
\]
(the space of \emph{$p$-dimensional cycles})
and
\[
B_p(K) = \mathrm{Im}(\partial_{p+1}(K))
\]
(the space of \emph{$p$-dimensional boundaries}).
\end{notation}

\begin{definition}[Simplicial homology groups]
\label{def:simplicial-homology}
The \emph{$p$-dimensional simplicial homology group $\HH_p(K)$} is defined as \[
\HH_p(K) = Z_p(K)/B_p(K).
\]
\end{definition}

\subsection{Persistent homology}	
We now define persistent homology (taken from \cite{basu-karisani-persistent}).

	Let $T$ be an ordered set, and
	$\mathcal{F} = (X_t)_{t \in T}$, a tuple of subspaces of $X$, such that
	$s \leq t \Rightarrow X_s \subset X_t$. We call $\mathcal{F}$ a filtration
	of the  topological space $X$.
	
	We now recall the definition of the \emph{persistent homology
		groups}  associated to a filtration
	\cite{Edelsbrunner_survey, Weinberger_survey}.
	%Unlike the rest of the paper, in this section 
	Since we only consider homology groups with rational coefficients, all homology groups in what follows are finite dimensional $\Q$-vector spaces.
	
	\begin{notation}
		\label{not:inclusion}
		For $s,t \in T, s \leq t$, and $p \geq 0$, we let $i_p^{s, t} : \HH_p (X_s) \longrightarrow
		\HH_p (X_t)$, denote the homomorphism induced by the inclusion $X_s
		\hookrightarrow X_t$.
	\end{notation}
	
	%%With the same notation as in the previous section we define:
	
	\begin{definition}
		\label{def:persistent}
		{\cite{Edelsbrunner_survey}} For each triple $(p, s, t)  \in \Z_{\geq 0} \times T \times T$ with 
		$
		%%0 \leq 
		s \leq t 
		%%\leq N
		$ 
		the   
		\emph{persistent homology  group}, $\HH_p^{s, t} (\mathcal{F})$ is defined by
		\begin{eqnarray*}
			\HH_p^{s, t} (\mathcal{F}) & = &  \mathrm{Im} (i_p^{s, t}).
		\end{eqnarray*}
		Note that $\HH_p^{s, t} (\mathcal{F}) \subset \HH_p (X_t)$, and $\HH_p^{s, s}
		(\mathcal{F}) = H_p (X_s)$.
	\end{definition}
	
	\begin{notation}
		We denote by $b_p^{s, t} (\mathcal{F}) = \dim_\Q (\HH_p^{s, t}
		(\mathcal{F}))$.
	\end{notation}

	Persistent homology measures how long a homology class persists in the filtration, in other words considering the homology classes as topological features, it gives an insight about the time  (thinking of the indexing set $T$ of the filtration as time) that a topological feature appears (or is born) and the time it disappears (or dies). This is made precise  as follows.  
	\begin{definition}
		For $s \leq t \in T$, and $ p \geq 0$,
		\begin{itemize}
			\item  we say that a homology class $\gamma \in \HH_p(X_{s})$ is \emph{born} at time $s$, if $\gamma \notin \HH_{p}^{s',s}(\mathcal{F})$, for any $s' < s$;
			\item for a class $\gamma \in \HH_p(S_{s})$ born at time $s$, we say that $\gamma$ \emph{dies} at time $t$,
			\begin{itemize}
				\item if  $i_p^{s,t'}(\gamma) \notin \HH_{p}^{s',t'}(\mathcal{F})$ for all $s', t'$ such that $s' < s \leq t' < t$,
				\item but $i_p^{s,t}(\gamma) \in \HH_{p}^{s'',t}(\mathcal{F})$, for some $s'' < s$.
			\end{itemize} 
		\end{itemize}  	
	\end{definition}
	
	\begin{remark}
		\label{rem:subspaces}
		Note that the homology 
		classes that are born at time $s$, and those that are born at time $s$ and dies at time $t$, 
		as defined above are not subspaces of $\HH_p(X_s)$. In order to be able to associate a ``multiplicity''
		to the set of homology classes which are born at time $s$ and dies at time $t$ we interpret them as classes in 
		certain \emph{subquotients}  of $\HH_*(X_s)$ in what follows.
	\end{remark}
	
	First observe that it follows from Definition~\ref{def:persistent} that for all $s' \leq s \leq t$ and $p \geq 0$, 
	$\HH_p^{s',t}(\mathcal{F})$ is a subspace of $\HH^{s,t}_p(\mathcal{F})$, and both are subspaces of
	$\HH_p(X_t)$. This is because the homomorphism $i^{s',t}_p = i^{s,t}_p\circ i^{s',s}_p$, and so the image of 
	$i^{s',t}_p$ is contained in the image of $i^{s,t}_p$.
	It follows that, for $s \leq t$,  the union of $\bigcup_{s' < s} \HH^{s',t}_p(\mathcal{F})$ is an increasing union of subspaces, and is 
	itself a subspace of $\HH_p(X_t)$.
	In particular, setting $t=s$,  $\bigcup_{s' < s} \HH^{s',s}(\mathcal{F})$ is a subspace of $\HH_p(X_s)$. \\
	
	With the same notation as  above:
	
	\begin{definition}[Subspaces of $\HH_p(X_s)$]
		\label{def:barcode_subspace}
		For $s \leq t$, and $p \geq 0$, we define
		\begin{eqnarray*}
			M^{s,t}_p(\mathcal{F}) &=& \bigcup_{s' < s} (i^{s,t}_p)^{-1}(\HH^{s',t}_p(\mathcal{F})), \\
			N^{s,t}_p(\mathcal{F}) &=& \bigcup_{s' < s\leq t' < t} (i^{s,t'}_p)^{-1}(\HH^{s',t'}_p(\mathcal{F})), \\
		\end{eqnarray*}
	\end{definition}

	\begin{remark}
		\label{rem:barcode_subspace_mean}
		The ``meaning''  of these subspaces are as follows.
		\begin{enumerate}[(a)] \label{rem:meaning} 
			\item
			For every fixed $s \in T$,  $M^{s,t}_p(\mathcal{F})$  is  a subspace of  $\HH_p(K_s)$ consisting of homology classes in $\HH_p(K_s)$ which are
			\begin{center}
				``born before time $s$, or born at time $s$ and dies at $t$ or earlier'' 
			\end{center}
			
			\item
			Similarly,
			for every fixed $s \in T$,  $N^{s,t}_p(\mathcal{F})$  is  a subspace of $\HH_p(K_s)$ consisting of homology classes in  $\HH_p(K_s)$ which are
			\begin{center}
				``born before time $s$, or born at time $s$ and dies strictly earlier than $t$''
			\end{center}
		\end{enumerate}
		
		%%The dimensions of $M^{s,t}_p(\mathcal{F})$ and $N^{s,t}_p(\mathcal{F})$ are given in Eqn. \eqref{eq:dim_M} and \eqref{eq:dim_N} in Proposition~\ref{prop:finite:multiplicity} below.  
	\end{remark}
	
	We now define certain \emph{subquotients} of the homology groups of $\HH_p(K_s), s \in T, p \geq 0$,
	in terms of the subspaces defined above in Definition~\ref{def:barcode_subspace}. 
	
	\begin{definition}[Subquotients associated to a filtration]
		\label{def:barcode_subquotients}
		For $s \leq t$, and $p \geq 0$, we define
		\begin{eqnarray*}
			P^{s,t}_p(\mathcal{F}) &=& M^{s,t}_p(\mathcal{F})/N^{s,t}_p(\mathcal{F}), \\
			P^{s,\infty}_p(\mathcal{F}) &=&  \HH_p(K_s) / \bigcup_{s \leq t} M^{s,t}_p(\mathcal{F}).
		\end{eqnarray*}
		
		We will call 
		\begin{enumerate}[(a)]
			\item
			$P^{s,t}_p(\mathcal{F})$ the \emph{space of $p$-dimensional cycles  born at time $s$ and which dies at time $t$};
			and 
			\item
			$P^{s,\infty}_p(\mathcal{F})$ the \emph{space of $p$-dimensional cycles born at time $s$ and which never die}.
		\end{enumerate}
	\end{definition}

	Finally, we define:
	
	\begin{definition}[Persistent multiplicity, barcode]
		\label{def:barcode_multiplicity}
		We will denote for $s \in T,  t \in T \cup \{\infty \}$,
		\begin{equation}
			\label{eqn:def:barcode:multiplicity}
			\mu^{s,t}_p(\mathcal{F}) = \dim P^{s,t}_p(\mathcal{F}),
		\end{equation}
		and call $\mu^{s,t}_p(\mathcal{F})$ the \emph{persistent multiplicity of $p$-dimensional cycles born at time $s$ and dying at time $t$ if $t \neq \infty$,   or never dying in case $t = \infty$}.

		We will call the set
		\begin{equation}
			\label{eqn:def:barcode}
			\mathcal{B}_p(\mathcal{F}) = \{(s,t,\mu^{s,t}_p(\mathcal{F})) \mid \mu^{s,t}_p(\mathcal{F}) > 0\}
		\end{equation}
		\emph{the $p$-dimensional barcode associated to the filtration $\mathcal{F}$}. 
		
		We will call an element
		$b = (s,t,\mu^{s,t}_p(\mathcal{F})) \in \mathcal{B}_p(\mathcal{F})$ a \emph{bar of $\mathcal{F}$ of multiplicity
			$\mu^{s,t}_p(\mathcal{F}$}). 
	\end{definition}

	%%%%%%%%%%%%%%%%%%%%%%%%%%%%%%%%%%%%%%%%%%%%%%%%%%%%%%%%%%%%%%%%%%%%%%%%
	
We now give a concrete example of a barcode associated to a 
(infinite) filtration (the following example is taken from \cite{Karisani2022}).

\begin{example}
	\label{eg:torus}
	Let $S$ be the two-dimensional torus 
	(topologically $\Sphere^1 \times \Sphere ^1$)
	embedded in $\mathbb{R}^3$, and $\mathcal{F}$ be
	the filtration of the torus by the sub-level sets of the height function  (depicted in Figure~\ref{Fig:tor}). 
	We denote by $S_{\leq t}$ the subset of the torus having ``height'' $\leq t$.
	
	We consider homology in dimensions  $0$, $1$ and $2$. 
	
	Informally, one observes that a $0$-dimensional homology class is born at time $t_0$ which never dies. There are two $1$-dimensional homology classes, the horizontal loop born at time $t_2$ and the vertical loop born at time $t_4$, which also never die. Lastly, there is a $2$-dimensional homology class born at time $t_5$ which never dies. Since there are no homology classes of the same dimension being born and dying at the same time, multiplicities in all the cases are 1. 
	
	More formally, following Definitions~\ref{def:barcode_subspace}, \ref{def:barcode_subquotients} and \ref{def:barcode_multiplicity}, we obtain:
	
	\begin{enumerate}[({Case p = }1)]
		\setcounter{enumi}{-1}
		\item If $t_0 \leq t < \infty$ then
		(using Definition~\ref{def:barcode_subspace}) 
		\[
		M^{t_0,t}_0(\mathcal{F}) = 0,
		\]
		and  hence
		(using Definitions~\ref{def:barcode_subquotients} and \ref{def:barcode_multiplicity})
		\[
		P^{t_0,t}_0(\mathcal{F}) = 0, \mbox{ and }
		\mu^{t_0, t}_0(\mathcal{F})=0.
		\]
		On the other hand,  
		\[
		P^{t_0,\infty}_0(\mathcal{F}) = \HH_0(S_{\leq t_0}),
		\]
		implying 
		\[
		\mu^{t_0, \infty}_0(\mathcal{F})=1.
		\]
		
		\item  For $t_2 \leq t < \infty$, 
		\[
		M^{t_2,t}_1(\mathcal{F}) = 0,
		\]
		and hence  
		\[P^{t_2,t}_1(\mathcal{F}) = 0, \mbox{ and }
		\mu^{t_2, t}_1(\mathcal{F})=0.
		\]
		Moreover,  
		\[
		P^{t_2,\infty}_1(\mathcal{F}) = \HH_1(S_{\leq t_2}),
		\]
		and therefore, 
		\[
		\mu^{t_2, \infty}_1(\mathcal{F})=1.
		\]
		
		\noindent For $t_4 \leq t < \infty$, 
		\[
		M^{t_4,t}_1(\mathcal{F}) = N^{t_4,t}_1(\mathcal{F}) =  \HH_1(S_{< t_4}),
		\]
		and hence 
		\[
		P^{t_4,t}_1(\mathcal{F}) = 0, \mbox{ and }
		\mu^{t_4, t}_1(\mathcal{F})=0.
		\]
		Moreover,  
		\[
		P^{t_4,\infty}_1(\mathcal{F}) =  \HH_1(S_{\leq t_4}) /  \HH_1(S_{< t_4}),
		\]
		and therefore  
		\[
		\mu^{t_4, \infty}_1(\mathcal{F})=1.
		\]
		
		\item For $t_5 \leq t < \infty$, 
		\[
		M^{t_5,t}_2(\mathcal{F}) = 0,
		\]
		and hence 
		\[
		P^{t_5,t}_2(\mathcal{F}) = 0, \mbox{ and }
		\mu^{t_5, t}_2(\mathcal{F})=0.
		\]
		Moreover, 
		\[
		P^{t_5,\infty}_2(\mathcal{F}) =\HH_2(S),
		\]
		and therefore 
		\[
		\mu^{t_5, \infty}_2(\mathcal{F})=1.
		\]
	\end{enumerate}
	Therefore the barcodes are as follows 
	(using Eqn.~\eqref{eqn:def:barcode}).
	
	$$ {\small
		\begin{array}{ccl}
			\mathcal{B}_0(\mathcal{F}) & = & \{(t_0, +\infty, 1)\}, \\
			\mathcal{B}_1(\mathcal{F}) & = & \{(t_2, +\infty, 1), (t_4, +\infty, 1)\},\\
			\mathcal{B}_2(\mathcal{F}) & = &\{(t_5, +\infty, 1)\}.
	\end{array}}
	$$ 
	Figure~\ref{Fig:bar} illustrates the corresponding bars.  Notice that even though the filtration $\mathcal{F}$ is an infinite filtration indexed by $\mathbb{R}$, the barcodes, $\mathcal{B}_p(\mathcal{F})$, are finite.
	
	\begin{figure}[h!]
		\centering
		\subfigure[]{%
			\label{Fig:tor}%
			\includegraphics[height=1.12in]{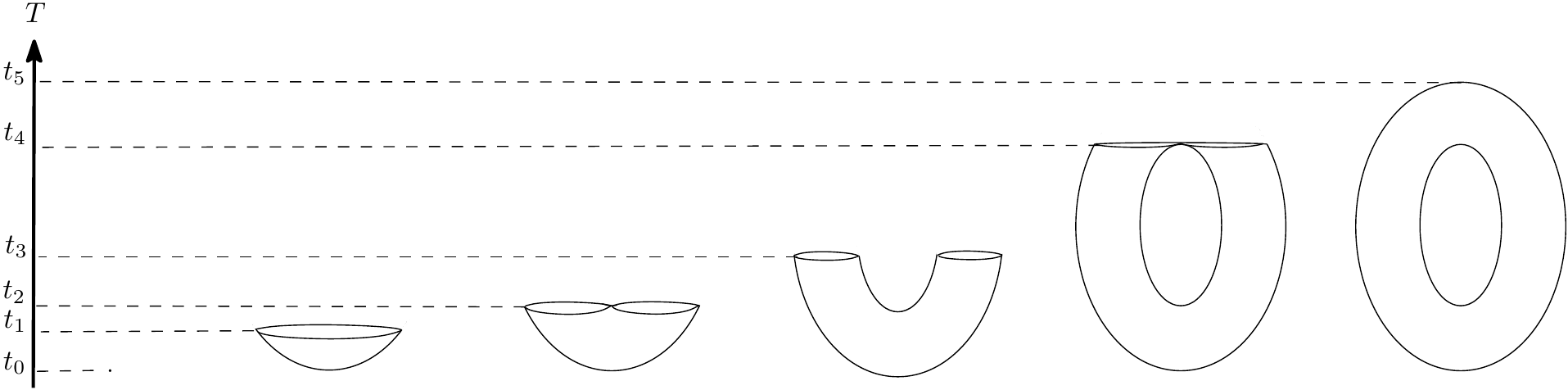}}%
		\qquad
		\subfigure[]{%
			\label{Fig:bar}%
			\includegraphics[height=1in]{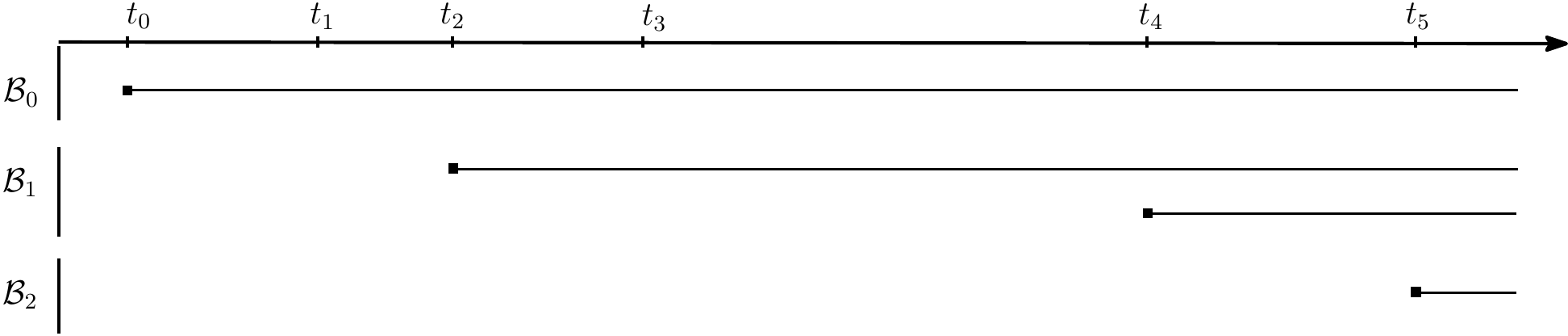}}%
		
		\caption{\small (a) Torus filtered by the sub-level sets of the height function, (b) corresponding barcodes for homology classes of dimension 0, 1 and 2. }
	\end{figure}
\end{example}
\end{document}